\documentclass[12pt, reqno]{amsart}
\usepackage{mathrsfs}
\usepackage{amsfonts}
\usepackage[centertags]{amsmath}
\usepackage{amssymb,array}
\usepackage{amsthm}

\usepackage{graphicx}
\usepackage{caption}

\usepackage{enumerate,enumitem}

\SetLabelAlign{center}{\hfill #1\hfill}
\usepackage[textwidth=16cm, hmarginratio=1:1]{geometry}
\usepackage{tikz-cd, tikz}
\usepackage{rotating}
\usepackage[all,cmtip]{xy}
\usepackage[titletoc, title]{appendix}
\usepackage{amscd}
\usepackage[colorlinks]{hyperref}
\usepackage{float}
 
\usepackage{diagbox}

\hypersetup{
pdftitle={Towards Monoidal Categorifications of Twisted Products of Flag Varieties},
pdfauthor={Yingjin Bi},
bookmarksnumbered,
pdfstartview={FitH},
breaklinks=true,
linkcolor=blue,
urlcolor=blue,
citecolor=blue,
bookmarksdepth=2
}
\usetikzlibrary{arrows.meta}

\theoremstyle{plain}
   \newtheorem{theorem}{Theorem}[section]
   
   \newtheorem{proposition}[theorem]{Proposition}
   
   \newtheorem{lemma}[theorem]{Lemma}
   \newtheorem{corollary}[theorem]{Corollary}
   \newtheorem{conjecture}[theorem]{Conjecture}
   
\theoremstyle{definition}
   \newtheorem{definition}[theorem]{Definition}
   
   \newtheorem{example}[theorem]{Example}

   \newtheorem{remark}[theorem]{Remark}

\numberwithin{equation}{section}

\newcommand{\be}{\begin{enumerate}}

    \newcommand{\ene}{\end{enumerate}}

    \newcommand{\ZZ}{\mathbb{Z}}

    \newcommand{\NN}{\mathbb{N}}

    \newcommand{\CC}{\mathbb{C}}

    \newcommand{\DD}{\mathbb{D}}

    \newcommand{\Id}{\operatorname{Id}}

    \newcommand{\fg}{\operatorname{\mathfrak{g}}}

    \newcommand{\bb}{\mathbf{b}}
    \newcommand{\bfa}{\mathbf{a}}

    \newcommand{\cA}{\mathcal{A}}

    \newcommand{\CM}[1]{\mathcal{M}}

    \newcommand{\Br}{{\operatorname{Br}}}
    \newcommand{\qbinom}[2]{\left[\begin{array}{c} #1 \\ #2 \end{array}\right]_q}

\newlength{\mysizetiny}
\newlength{\mysizesmall}
\newlength{\mysize}
\newlength{\mysizelarge}
\begin{document}

\title[Towards categorifications of twisted products]{Towards Monoidal Categorifications of Twisted Products of Flag Varieties}
\author{Yingjin Bi}
\address{Department of Mathematics, Harbin Engineering University}
\email{yingjinbi@mail.bnu.edu.cn}
\date{}

\begin{abstract}
Let $G$ be a simple, simply connected algebraic group of simply-laced type.
For a positive braid word $\beta$ and $v\le\delta(\beta)$, we study the
cluster algebra associated with the twisted product of flag varieties
$\mathring{\mathcal Z}_{v,\beta}$. We compare its Bao--Ye seed with a
right-inductive weave seed and obtain local acyclicity and equality of
the cluster and upper cluster algebras. Using Lusztig parameters in a
bosonic extension algebra, we construct a monoidal subcategory
$\mathscr C_{v,\beta}$ of a Hernandez--Leclerc category and prove that
its Grothendieck ring contains the integral cluster algebra with
noninvertible frozen variables. Every cluster monomial is the class of
a real simple object of $\mathscr C_{v,\beta}$. The reverse inclusion,
which would give a full monoidal categorification, is left as a conjecture.
\end{abstract}

\maketitle
\tableofcontents

\section{Introduction}

Cluster algebras, introduced by Fomin and Zelevinsky
\cite{fomin2002cluster}, connect the geometry of Lie-theoretic varieties
with representation theory and canonical bases. On the geometric side,
coordinate rings of double Bruhat cells, double Bott--Samelson cells,
and braid varieties admit cluster structures
\cite{shen2021cluster,casals2025cluster,galashin2025braid}.
On the representation-theoretic side, monoidal categorification realizes
cluster variables and cluster monomials by simple modules
\cite{hernandez2010cluster,kang2018monoidal,kashiwara2024monoidal,
kashiwara2025monoidal}.

Twisted products of flag varieties provide a common framework for many
of these examples. Bao and Ye \cite{bao2025upper} constructed upper
cluster structures on their open Richardson strata, extending the
mutation constructions for open Richardson varieties developed by
M\'enard \cite{menard2022cluster}. They also proposed comparing their
braid-variety seeds with the weave seeds
\cite[Section~1.2]{bao2025upper}. We establish this comparison in
simply-laced finite type and study its representation-theoretic realization
inside a fixed Hernandez--Leclerc category.

The paper has three main contributions. Lemma~\ref{lem:inductivewave}
identifies the Bao--Ye seed with an appropriate right-inductive weave
seed, yielding local acyclicity and equality of the cluster and upper
cluster algebras. Proposition~\ref{pro:bfadotbeta} and
Theorem~\ref{thm:induction} compute the selected Lusztig coordinates
through the prescribed mutation sequence. Finally,
Theorem~\ref{thm:categorification} embeds the integral cluster algebra
in the Grothendieck ring of an intersection subcategory
$\mathscr C_{v,\beta}$, with every cluster monomial represented by a
real simple object. Surjectivity is formulated as
Conjecture~\ref{conj:full-categorification}.

The representation-theoretic argument uses the bosonic extension
algebra and the categorical PBW theory of Kashiwara--Kim--Oh--Park
\cite{kashiwara2024braid,kashiwara2025global,kashiwara2025monoidal}.
The decisive invariant is the initial segment of a Lusztig parameter
computed in a word beginning with a reduced expression of $v$.
Membership in the relevant tail subcategory is characterized by the
vanishing of these coordinates. We prove this vanishing for the
Bao--Ye initial variables by following the prescribed mutation
sequence. In this argument the auxiliary quiver records only variables
with nonzero truncated parameter; factors omitted from that quiver are
retained in the full exchange relation. The vanishing then propagates
along arbitrary mutations of the resulting seed.

We now describe the geometric, algebraic, and categorical objects more
precisely. Let $G$ be a simple, simply connected complex algebraic group
of simply-laced type, with opposite Borel subgroups $B^+$ and $B^-$.
Let $I$ index the simple roots, let $W$ be the Weyl group, and let
$\operatorname{Br}^+$ be the positive braid monoid. For an expression
\[
\beta=(i_1,\ldots,i_r)
\]
of $b\in\operatorname{Br}^+$, its Demazure product is denoted by
$\delta(b)=\delta(\beta)$. Thus $\delta(1)=e$ and
\[
\delta(\sigma_i b)=\max\{\delta(b),s_i\delta(b)\},
\]
where the maximum is taken in Bruhat order.

For $v\le\delta(b)$, let
\[
\beta_v=(i_{p_1},\ldots,i_{p_m}),\qquad m=\ell(v),
\]
be the leftmost reduced subexpression of $v$ in $\beta$: the increasing
sequence $(p_1,\ldots,p_m)$ is lexicographically minimal among reduced
subexpressions representing $v$. We use the empty sequence when $v=e$.

Consider the open twisted product
\[
\mathring{\mathcal Z}_\beta
=
\bigl(B^+\dot s_{i_1}B^+\times^{B^+}\cdots
\times^{B^+}B^+\dot s_{i_r}B^+\bigr)/B^+,
\]
and its multiplication map
\[
m_\beta:\mathring{\mathcal Z}_\beta\longrightarrow G/B^+,
\qquad (g_1,\ldots,g_r)\longmapsto g_1\cdots g_rB^+.
\]
We set
\[
\mathring{\mathcal Z}_{v,\beta}
:=m_\beta^{-1}(B^-\dot vB^+/B^+).
\]
These varieties include braid varieties and reduced double Bruhat
cells. In particular,
$\mathring{\mathcal Z}_{\delta(b),\beta}\simeq X(\beta)$;
see Proposition~\ref{pro:braid}. More generally, the geometric
isomorphisms established in Section~\ref{sec:twisted} identify
$\mathring{\mathcal Z}_{v,\beta}$ with a braid variety obtained by
adjoining a reduced word for $w_0v^{-1}$ to the left of $\beta$.

Let $\widehat{\mathcal A}$ be the bosonic extension algebra, with
formal deformation parameter $q$, and let $T_i$ be its braid
symmetries. For $b\in\operatorname{Br}^+$, let
$\widehat{\mathcal A}(b)$ be the associated braid subalgebra. We define
\[
\widehat{\mathcal A}_{v,b}
:=\widehat{\mathcal A}(b)\cap T_v(\widehat{\mathcal A}_{\ge0}),
\]
where $T_v$ is defined by any reduced expression of $v$.
This algebra depends only on $v$ and $b$.
Proposition~\ref{prop:Avb-characterization} describes its global basis
in terms of vanishing initial Lusztig coordinates.

For the categorical realization, fix a generic parameter
$\zeta\in\mathbb C^*$ and a complete duality datum
$\mathbb D=(L_i^{\mathbb D})_{i\in I}$ in the Hernandez--Leclerc
category $\mathscr C^{\mathbb Z}$ for
$U_\zeta'(\widehat{\mathfrak g})$.
The word $\beta$ determines affine cuspidal modules
$C_k^{\mathbb D,\beta}$, $1\le k\le r$.
Let $\mathscr C_{\mathbb D}(\beta)$ be the full subcategory they
generate under tensor products, extensions, and subquotients.
When $\mathbb D$ is fixed, we abbreviate it to $\mathscr C(\beta)$.

Extend $\beta_v$ to a reduced expression
$\overline w_0=(j_1,\ldots,j_N)$ of the longest element, and form
\begin{equation}\label{eq:dotw0}
\dot w_0
=(j_1,\ldots,j_N,j_1^*,\ldots,j_N^*,j_1,\ldots,j_N,\ldots),
\end{equation}
where $w_0(\alpha_j)=-\alpha_{j^*}$. We also denote this infinite word
by $\dot\beta_v$. Let $\mathscr C^v$ be the full subcategory generated
under tensor products, extensions, and subquotients by
\[
C_k^{\mathbb D,\dot\beta_v},\qquad k>\ell(v),
\]
and set
\[
\mathscr C_{v,\beta}:=\mathscr C(\beta)\cap\mathscr C^v.
\]
The intersection is taken inside $\mathscr C^{\mathbb Z}$.
Proposition~\ref{prop:categorical-tail-intrinsic} shows that
$\mathscr C^v$ is independent of the reduced expression and completion
used above; in particular, $\mathscr C_{v,\beta}$ depends only on
$\mathbb D,v,b$.
The categorical parameter criterion in Section~\ref{sec:category}
relates this intersection to $\widehat{\mathcal A}_{v,b}$ without
assuming a quantum Grothendieck-ring isomorphism for an arbitrary
complete duality datum.

Let $\mathbf s(v,\beta)$ be the seed obtained from $\mathbf s(\beta)$
by the mutation, freezing, and deletion procedure of
Definition~\ref{def:svbeta}. We write $\mathcal A_0(\mathbf s(v,\beta))$ for
the integral cluster algebra with noninvertible frozen variables,
and use the subscript $\mathrm{loc}$ to indicate localization at
those variables. Our main result is the following.

\begin{theorem}[{Theorems~\ref{thm:categorification} and~\ref{thm:clusterproduct}}]
\label{thm:main-intro}
Let $b\in\operatorname{Br}^+$, let $v\le\delta(b)$, and let $\beta$
be an expression of $b$. For every complete duality datum $\mathbb D$,
there is an injective homomorphism of $\mathbb Z$-algebras
\[
\mathcal A_0(\mathbf s(v,\beta))
\hookrightarrow K_0(\mathscr C_{v,\beta})
\]
under which every cluster monomial is the class of a real simple object.
Moreover, the Bao--Ye cluster structure is locally acyclic and
\[
\mathcal A_0(\mathbf s(v,\beta))_{\mathrm{loc}}
\otimes_{\mathbb Z}\mathbb C
\simeq \mathbb C[\mathring{\mathcal Z}_{v,\beta}].
\]
\end{theorem}

The theorem establishes the inclusion needed for a monoidal
categorification. The remaining direction is the following conjecture,
where the two rings are compared through the displayed embedding.
\begin{conjecture}\label{conj:full-categorification}
One has
\[
\mathcal A_0(\mathbf s(v,\beta))=K_0(\mathscr C_{v,\beta}).
\]
\end{conjecture}
We also expect a suitable integral form of
$\widehat{\mathcal A}_{v,b}$, after localization at the frozen
variables, to give a quantization of
$\mathbb C[\mathring{\mathcal Z}_{v,\beta}]$.
No equality with this entire intersection algebra, or flat
specialization theorem for it, is asserted here.

\medskip
\noindent\textbf{Organization of the paper.}
Section~\ref{sec:pre} fixes the word, seed, and coefficient conventions.
Section~\ref{sec:twisted} studies the geometric realizations and compares
the cluster seeds. Section~\ref{sec:bosonic} recalls the bosonic
extension, its normalized global basis, and its quantum cluster
structure. Section~\ref{sec:subalgebra} proves the parameter criterion
and the mutation induction. Section~\ref{sec:category} transfers these
results to the intersection subcategory and proves the categorical
part of Theorem~\ref{thm:main-intro}.

\medskip
\noindent\textbf{Acknowledgements.}
The author is deeply grateful to Masaki Kashiwara for proposing the
theme of this work and for his guidance and encouragement throughout
the project. The author thanks Ryo Fujita for many helpful discussions
and for the invitation to RIMS, which provided an excellent environment
for this research, and Huanchen Bao for helpful discussions on twisted
products of flag varieties. The author is supported by the China
Scholarship Council. Grant No. 202506680034.

\section{Preliminaries}\label{sec:pre}

Let \(Q=(I,\Omega)\) be a Dynkin quiver with vertex set \(I\), and let
\(C=(c_{ij})_{i,j\in I}\) be the corresponding Cartan matrix.
Let $\Br$ and $W$ denote the braid group and the Weyl group associated with $Q$,
generated by $\{\sigma_i\}_{i\in I}$ and $\{s_i\}_{i\in I}$, respectively.
For a nonnegative integer $r$, we write $[r]:=\{1,\dots,r\}$, with $[0]=\varnothing$. Denote by $d(i,j)$ the number of edges connecting vertex $i$ with vertex $j$ in $Q$. 

We denote by $R^+$ the set of positive roots, by $Q$ and $Q^+$ the root
lattice and its positive cone, and by $\alpha_i$ and $\alpha_i^\vee$ the
simple roots and coroots. The fundamental weights are denoted by
$\varpi_i$, and $P=\bigoplus_{i\in I}\mathbb Z\varpi_i$ is the weight
lattice. We use the $W$-invariant bilinear form
$(\cdot,\cdot):P\times P\to\mathbb Q$ normalized by
$(\alpha_i,\alpha_j)=c_{ij}$. Its restriction to $Q\times P$ is
integer-valued.

\subsection{Words of vertices}\label{subsec:words}

Let \(b\in \Br^+\), and fix an expression
\(
\beta=(i_1,\dots,i_r)
\)
of \(b\), that is,
\(
b=\sigma_{i_1}\cdots \sigma_{i_r}.
\)
We say that the letter \(i_k\) has \emph{color} \(j\in I\) if \(i_k=j\).
For each \(j\in I\), let \(n_j\) be the number of occurrences of \(j\) in \(\beta\).

For \(k\in[1,r]\), define
\[
N_j(k):=\#\{p\le k\mid i_p=j\}.
\]
Thus, if \(i_k=j\) and \(N_j(k)=n\), then the position \(k\) is the
\(n\)-th occurrence of the color \(j\) in \(\beta\). In this case, we sometimes
write
\[
k\leftrightarrow (j,n).
\]

For a position \(k\in[1,r]\), set
\[
k^{\min}:=\min\{p\in[1,r]\mid i_p=i_k\},
\qquad
k^{\max}:=\max\{p\in[1,r]\mid i_p=i_k\}.
\]
We also define
\[
k^+:=\min\{p>k\mid i_p=i_k\},
\qquad
k^-:=\max\{p<k\mid i_p=i_k\},
\]
with the convention that \(k^+=+\infty\) if no such \(p\) exists, and
\(k^-=-\infty\) if no such \(p\) exists. More generally, for \(j\in I\), define
\[
k(j)^+ := \min\{p>k\mid i_p=j\},
\qquad
k(j)^- := \max\{p<k\mid i_p=j\},
\]
with the same convention.

For \(b\in \Br^+\), we define the \emph{Demazure product}
\(\delta(b)\in W\) inductively by
\[
\delta(1)=e,\qquad
\delta(\sigma_i b)=\max\{\delta(b),s_i\delta(b)\},
\]
where the maximum is taken with respect to the Bruhat order on \(W\).

Let \(v\le \delta(b)\), and let \(m=\ell(v)\). By the subword property for the
Weyl group, there exists a subsequence
\(
1\le q_1<\cdots<q_m\le r
\)
such that
\(
v=s_{i_{q_1}}\cdots s_{i_{q_m}}.
\)
Among all such subsequences, let
\(
p_1<\cdots<p_m
\)
be the lexicographically minimal one. We call
\[
\beta_v=(i_{p_1},\dots,i_{p_m})
\]
the \emph{leftmost subexpression} of \(\beta\) associated with \(v\).

For \(k\in[m]\), define
\[
a_k:=\#\{s\le k\mid i_{p_s}=i_{p_k}\},
\]
and
\[
b_k:=\#\{t\le p_k\mid t\notin\{p_1,\dots,p_m\},\ i_t=i_{p_k}\}.
\]
Equivalently,
\[
d_k:=a_k+b_k
     =\#\{t\le p_k\mid i_t=i_{p_k}\}.
\]
Thus, if \(j=i_{p_k}\), then the position \(p_k\) is the \(d_k\)-th occurrence
of the color \(j\) in the word \(\beta\), that is,
\[
p_k\leftrightarrow (j,d_k).
\]

For \(k\in[m]\), define
\[
k^{\oplus}:=\min\{k<s\leq m \mid i_{p_s}=i_{p_k}\},
\]
with the convention that \(k^{\oplus}=+\infty\) if no such \(s\) exists.
More generally, for \(j\in I\) and \(0\le k\le m\), set
\[
k(j)^{\oplus}:=\min\{k<s\leq m\mid i_{p_s}=j\},
\]
again with the convention that \(k(j)^{\oplus}=+\infty\) if the set is empty.

Finally, for \(j\in I\) and \(0\le k\le m\), define
\[
\alpha(j,k):=\#\{s\le k\mid i_{p_s}=j\}.
\]
In particular,
\[
\alpha(i_{p_k},k)=a_k.
\]
We set $d_{+\infty}=+\infty$ and $\alpha(j,0)=0$. If $v=e$, the
selected word is empty and every mutation sequence indexed by $[m]$ is empty.
\begin{example}\label{exm:words}
Let us consider type $A_3$ and the word
\[
\beta := (3,2,1,2,3,1,3,2) = (i_1,\ldots,i_8).
\]
Let
\(
v = s_3 s_2 s_3 s_1 s_2.
\)
Then we have $5 \leftrightarrow (3,2)$, with
\[
5^{+}=7,\qquad 5^{-}=1,
\]
and
\[
5(2)^{+}=8,\qquad 5(2)^{-}=4.
\]
The Demazure product satisfies $\delta(\beta)=w_0$.
The leftmost reduced subexpression corresponding to $v$ is
\[
(\underline{3},\underline{2},1,2,\underline{3},\underline{1},3,\underline{2}).
\]
That is,
\[
\beta_v=(i_1,i_2,i_5,i_6,i_8)=(j_1,\ldots,j_5).
\]

Let $k=5\in[5]$ and $p_5=8$. Then we have $a_5=2$ and $b_5=1$, and
\[
8\leftrightarrow(2,3).
\]
For $3\in [5]$, we have $\alpha(2,3)=1, \alpha(3,3)=2$, and $3^{\oplus}=+\infty, 3(2)^{\oplus}=5$. 
\end{example}

\subsection{Cluster algebras}

Let \(Q=(K,Q_1)\) be a quiver without loops or 2-cycles, and let
\[
K=K_{\operatorname{ex}}\sqcup K_{\operatorname{fr}}
\]
be a decomposition of its set of vertices into exchangeable and frozen vertices.
We associate to \(Q\) the integer matrix
\begin{equation}\label{eq_exchange_matrix}
B_Q=(b_{ij})_{i\in K,\ j\in K_{\operatorname{ex}}},
\qquad
b_{ij}=\#\{j\to i\}-\#\{i\to j\}.
\end{equation}
Let \(L=(\lambda_{ij})_{i,j\in K}\) be a skew-symmetric integer matrix. We say
that \(L\) is compatible with \(B_Q\) if
\[
\sum_{k\in K} b_{ki}\lambda_{kj}=2\delta_{ij}
\qquad
\text{for all } i\in K_{\operatorname{ex}},\ j\in K.
\]

Let \(\mathfrak F\) be an ambient field and $x_i\in \mathfrak F$ for all $i\in K$. A \(\Lambda\)-seed in \(\mathfrak F\) is a quadruple
\[
\mathbf s=(\{x_i\}_{i\in K},L,B_Q,K_{\operatorname{ex}})
\]
such that:
\begin{enumerate}
\item the elements \(\{x_i\}_{i\in K}\) are algebraically independent over
\(\mathbb Q\);
\item \(L\) is compatible with \(B_Q\).
\end{enumerate}
The set \(\{x_i\}_{i\in K}\) is called the cluster of \(\mathbf s\). Its elements
are called cluster variables, and the variables \(x_i\) with
\(i\in K_{\operatorname{fr}}\) are called frozen variables.

For \(\mathbf a=(a_i)_{i\in K}\in \mathbb Z^K\), we write
\[
x^{\mathbf a}:=\prod_{i\in K}x_i^{a_i}.
\]
If \(\mathbf a\in \mathbb Z_{\ge 0}^K\), then \(x^{\mathbf a}\) is called a
cluster monomial.

Let \(k\in K_{\operatorname{ex}}\). The mutation \(\mu_k(\mathbf s)\) is defined as follows.
The mutated exchange matrix \(B'=\mu_k(B_Q)\) is given by
\[
b'_{ij}
=
\begin{cases}
-b_{ij}, & i=k \text{ or } j=k,\\
b_{ij}+[b_{ik}]_+[b_{kj}]_+
-[-b_{ik}]_+[-b_{kj}]_+,
& \text{otherwise},
\end{cases}
\]
where \([a]_+=\max(a,0)\). The mutated skew-symmetric matrix
\(L'=\mu_k(L)\) is given by
\[
\lambda'_{ij}
=
\begin{cases}
-\lambda_{kj}+\displaystyle\sum_{t\in K}[-b_{tk}]_+\lambda_{tj},
& i=k,\ j\ne k,\\[6pt]
-\lambda_{ik}+\displaystyle\sum_{t\in K}[-b_{tk}]_+\lambda_{it},
& i\ne k,\ j=k,\\[6pt]
\lambda_{ij}, & \text{otherwise}.
\end{cases}
\]
Finally, the mutated cluster variables are
\[
x'_i=
\begin{cases}
x^{\mathbf a'}+x^{\mathbf a''}, & i=k,\\
x_i, & i\ne k,
\end{cases}
\]
where
\[
a'_i=
\begin{cases}
-1, & i=k,\\
[b_{ik}]_+, & i\ne k,
\end{cases}
\qquad
a''_i=
\begin{cases}
-1, & i=k,\\
[-b_{ik}]_+, & i\ne k.
\end{cases}
\]
Then
\[
\mu_k(\mathbf s)
=
(\{x'_i\}_{i\in K},L',B',K_{\operatorname{ex}})
\]
is again a \(\Lambda\)-seed.

We use an integral form for statements about Grothendieck rings.
Let \(\cA_0(\mathbf s)\) be the \(\mathbb Z\)-subalgebra of
\(\mathfrak F\) generated by all cluster variables, including the
frozen variables, without adjoining their inverses. Let \(S_{\rm fr}\)
be the multiplicative set generated by the frozen variables, and set
\[
\cA_0(\mathbf s)_{\rm loc}:=S_{\rm fr}^{-1}\cA_0(\mathbf s),
\qquad
\cA(\mathbf s):=\mathbb Q\otimes_{\mathbb Z}
\cA_0(\mathbf s)_{\rm loc}.
\]
Thus \(\cA(\mathbf s)\) is the cluster algebra over \(\mathbb Q\)
with all frozen variables inverted. Complex coordinate rings will be
compared with \(\cA(\mathbf s)\otimes_{\mathbb Q}\mathbb C\), or
equivalently with
\(\cA_0(\mathbf s)_{\rm loc}\otimes_{\mathbb Z}\mathbb C\).

The upper cluster algebra is
\[
U(\mathbf s)
=
\bigcap_{\mathbf t\in T}
\mathbb Q[x_{\mathbf t,i}^{\pm1}\mid i\in K],
\]
where \(T\) is the set of all seeds mutation-equivalent to \(\mathbf s\), and
the intersection is taken inside \(\mathfrak F\).

Assume now that a total order on \(K\) is fixed. Let
\[
\mathbb K_t=\mathbb Z[t^{\pm1/2}].
\]
For a skew-symmetric integer matrix \(L=(\lambda_{ij})_{i,j\in K}\), the
quantum torus \(\mathcal T_L\) is the \(\mathbb K_t\)-algebra generated by
\(X_i^{\pm1}\), \(i\in K\), with relations
\[
X_iX_j=t^{\lambda_{ij}}X_jX_i,
\qquad
X_iX_i^{-1}=X_i^{-1}X_i=1.
\]
For \(\mathbf a=(a_i)_{i\in K}\in\mathbb Z^K\), define the normalized monomial
\[
X^{\mathbf a}
=
t^{\frac12\sum_{i>j}a_ia_j\lambda_{ij}}
\prod_{i\in K}X_i^{a_i},
\]
where the product is taken with respect to the fixed order on \(K\).

For a quantum seed
\[
\mathbf s=(\{X_i\}_{i\in K},L,B_Q,K_{\operatorname{ex}})
\]
and \(k\in K_{\operatorname{ex}}\), the mutation is defined by the same formulas
as above, with the exchange relation
\[
X'_k=X^{\mathbf a'}+X^{\mathbf a''},
\qquad
X'_i=X_i\quad (i\ne k),
\]
where the monomials are normalized quantum monomials.

The \emph{quantum cluster algebra} \(\mathcal A_t(\mathbf s)\) is the
\(\mathbb K_t\)-subalgebra of the ambient skew-field of \(\mathcal T_L\)
generated by all quantum cluster variables appearing in seeds
mutation-equivalent to \(\mathbf s\), together with the inverses
\(X_f^{-1}\) of the frozen variables for all \(f\in K_{\operatorname{fr}}\).

If the frozen variables are not inverted, we denote the corresponding
\(\mathbb K_t\)-subalgebra by \(\overline{\mathcal A}_t(\mathbf s)\).
Thus \(\overline{\mathcal A}_t(\mathbf s)\) is generated by all quantum
cluster variables appearing in seeds mutation-equivalent to \(\mathbf s\),
but not by the inverses of the frozen variables.

The bar anti-involution fixes every \(X_i\), sends
\(t^{1/2}\) to \(t^{-1/2}\), and reverses products. The normalized
monomials above are bar-invariant. In Section~\ref{sec:bosonic}, the
bosonic deformation parameter is denoted by \(q\). The two parameters
are related by
\begin{equation}\label{eq:quantum-parameter-convention}
\mathbb K_t\longrightarrow\mathbb K:=\mathbb Z[q^{\pm1/2}],
\qquad t^{1/2}\longmapsto q^{-1/2}.
\end{equation}
In particular, a realization of the quantum seed in the bosonic extension
satisfies \(D_iD_j=q^{-\lambda_{ij}}D_jD_i\).
We distinguish this scalar identification from specialization at
\(t^{1/2}=1\), which gives the commutative seed. We use the same vertex
labels and exchange matrix in these two realizations.

\subsubsection{Seed of words}\label{subsec:seed_words}

Let \(\beta=(i_1,\dots,i_r)\) be an expression of a positive braid
\(b\in \Br^+\), that is,
\[
b=\sigma_{i_1}\cdots \sigma_{i_r}.
\]
Set \(K=[r]\). We use the occurrence notation introduced above. Thus, if
\(i_k=j\) and \(k\) is the \(n\)-th occurrence of \(j\) in \(\beta\), we may
write \(k\leftrightarrow(j,n)\).

We define the set of frozen vertices by
\[
K_{\mathrm{fr}}
:=
\{k\in K\mid k^+=+\infty\}.
\]
Equivalently, under the occurrence notation,
\[
K_{\mathrm{fr}}
=
\{(j,n_j)\mid j\in I,\ n_j>0\}.
\]
We set
\[
K_{\mathrm{ex}}:=K\setminus K_{\mathrm{fr}}.
\]

We next recall the exchange matrix associated with \(\beta\). Define
\[
B_{\beta}=(b_{kl})_{k\in K,\ l\in K_{\mathrm{ex}}}
\]
by
\begin{equation}
\label{eq_braidmatrix}
b_{kl}=
\begin{cases}
1,
& \text{if } l=k^-,
\\
-1,
& \text{if } l=k^+,
\\
c_{i_k i_l},
& \text{if } l<k<l^+<k^+,
\\
-c_{i_k i_l},
& \text{if } k<l<k^+<l^+,
\\
0,
& \text{otherwise}.
\end{cases}
\end{equation}
Here the inequalities are understood with the conventions
\(k^-=-\infty\) if \(k^-\) does not exist and \(k^+=+\infty\) if \(k^+\)
does not exist.

Define
\[
L_\beta=(\lambda_{kl})_{k,l\in K}
\]
as follows. Let
\[
w_s=s_{i_1}\cdots s_{i_s}
\qquad (s\in [r]).
\]
We set \(\lambda_{kk}=0\), and for \(k<l\), define
\[
\lambda_{kl}
=
-\bigl(
\varpi_{i_k}-w_k\varpi_{i_k},
\varpi_{i_l}+w_l\varpi_{i_l}
\bigr),
\qquad
\lambda_{lk}=-\lambda_{kl}.
\]
By \cite[Proposition~1.2]{fujita2023isomorphisms},  the pair
\((B_\beta,L_\beta)\) is compatible, see Remark \ref{rem_BGLS}.

The ice quiver \(Q_\beta\) is the quiver whose extended exchange matrix is
\(B_\beta\). Thus \(Q_\beta\) has horizontal arrows
\[
k\longrightarrow k^+
\]
whenever \(k^+\) exists. It also has ordinary arrows
\[
k\longrightarrow l
\]
whenever
\[
l<k<l^+<k^+
\qquad\text{and}\qquad
c_{i_k i_l}=-1.
\]
\begin{remark}
\label{rem_BGLS}
Both signs in a compatible pair must be changed together when reversing
all arrows. We retain the matrices displayed above and the compatibility
condition \(B_\beta^{\mathsf T}L_\beta=2I\) on exchangeable columns.
Their bosonic realization uses the parameter map
\eqref{eq:quantum-parameter-convention}; in particular the exponent in
\(D_kD_l=q^{-\lambda_{kl}}D_lD_k\) is \(-\lambda_{kl}\).
This convention will be used in every quantum exchange relation.
\end{remark}

\begin{example}\label{exm:quiver}
Following Example~\ref{exm:words}, the quiver \(Q_\beta\) is shown in
Figure~\ref{fig:Qbeta}.
\begin{figure}[htbp]
\centering
\begin{tikzpicture}[
  x=0.85cm,
  y=0.78cm,
  >={Stealth[length=1.5mm]},
  every node/.style={font=\scriptsize},
  every edge/.style={draw, ->, thick}
]

\node (1) at (7,0) {$1$};
\node (2) at (6,1) {$2$};
\node (3) at (5,2) {$3$};
\node (4) at (4,1) {$4$};
\node (5) at (3,0) {$5$};
\node (6) at (2,2) {$6$};
\node (7) at (1,0) {$7$};
\node (8) at (0,1) {$8$};

\draw (3) edge (6);
\draw (2) edge (4);
\draw (4) edge (8);
\draw (1) edge (5);
\draw (5) edge (7);

\draw (3) edge (2);
\draw (4) edge (3);
\draw (6) edge (4);
\draw (4) edge (1);
\draw (7) edge (4);

\end{tikzpicture}
\caption{The quiver \(Q_{\beta}\).}
\label{fig:Qbeta}
\end{figure}
\end{example}

We define
\[
\mathbf{s}(\beta)
:=
\bigl(\{X_k\}_{k\in K},\,L_\beta,\,B_\beta,\,K_{\mathrm{ex}}\bigr)
\]
to be the associated quantum \(\Lambda\)-seed.

\begin{definition}
For a word \(\beta\), we denote by
\[
\mathcal{A}(\mathbf{s}(\beta)),\qquad
U(\mathbf{s}(\beta)),\qquad
\mathcal{A}_t(\mathbf{s}(\beta)),\qquad
\overline{\mathcal{A}}_t(\mathbf{s}(\beta))
\]
the cluster algebra, the upper cluster algebra, the quantum cluster algebra
with frozen variables inverted, and the quantum cluster algebra without
inverting frozen variables, respectively, associated with the seed
\(\mathbf{s}(\beta)\).

The realization by the right-inductive weave for
\(\Delta\beta\), together with
\cite[Theorem~7.13]{casals2025cluster}, gives
\[
\mathcal{A}(\mathbf{s}(\beta))=U(\mathbf{s}(\beta)).
\]
\end{definition}

\subsubsection{Leftmost subexpressions}

Let \(v\le \delta(b)\), and set \(m=\ell(v)\). Let
\[
\beta_v=(i_{p_1},\dots,i_{p_m})
\]
be the leftmost reduced subexpression of \(\beta\) representing \(v\). Thus
\[
1\le p_1<\cdots<p_m\le r,
\qquad
v=s_{i_{p_1}}\cdots s_{i_{p_m}},
\]
and the sequence \((p_1,\dots,p_m)\) is lexicographically minimal among all
reduced subexpressions of \(\beta\) representing \(v\).

Set
\[
p_0=0,\qquad p_{m+1}=r+1.
\]
For \(0\le k\le r\), define
\[
v_k:=s_{i_{p_1}}\cdots s_{i_{p_t}}
\qquad
\text{if } p_t\le k<p_{t+1}.
\]
In particular, \(v_k=e\) for \(0\le k<p_1\), and \(v_k=v\) for
\(p_m\le k\le r\).

We also define \(v'_k\) inductively by setting \(v'_0=e\) and
\begin{equation}\label{eq:vk'}
v'_k=
\begin{cases}
v'_{k-1}s_{i_k},
& \text{if } s_{i_k}(v'_{k-1})^{-1}v < (v'_{k-1})^{-1}v,\\
v'_{k-1},
& \text{otherwise}.
\end{cases}
\end{equation}

\begin{lemma}\label{lem:vk}
Let \(\beta=(i_1,\dots,i_r)\) be an expression of \(b\), and let
\(v\le \delta(b)\). Then
\[
v_k=v'_k
\]
for all \(0\le k\le r\).
\end{lemma}

\begin{proof}
The case \(v=e\) is immediate. Suppose that \(v'_{k-1}=v_{k-1}\),
and let \(t=\#\{s\mid p_s<k\}\). The residual element is
\[
 z:=v_{k-1}^{-1}v=s_{i_{p_{t+1}}}\cdots s_{i_{p_m}},
\]
with the empty product interpreted as \(e\).
If \(k=p_{t+1}\), the displayed expression is reduced and begins
with \(s_{i_k}\), so \(s_{i_k}z<z\). Thus both constructions append
\(s_{i_k}\).

Otherwise \(k<p_{t+1}\), or \(t=m\). In the latter case \(z=e\)
and neither construction changes the prefix. In the former case,
suppose that \(s_{i_k}z<z\). By the exchange condition,
\(s_{i_k}z\) has a reduced expression \(\gamma\) obtained by deleting
one letter from \((i_{p_{t+1}},\ldots,i_{p_m})\). Therefore
\[
 (i_{p_1},\ldots,i_{p_t},i_k,\gamma)
\]
is a reduced subexpression for \(v\). Its index sequence is
lexicographically smaller than \((p_1,\ldots,p_m)\), contrary to
leftmostness. Hence \(s_{i_k}z>z\), and both constructions leave
the prefix unchanged. Induction from \(v'_0=v_0=e\) proves the claim.
\end{proof}

\subsubsection{A sequence of mutations}

In this subsection, we introduce a sequence of mutations which will play an
important role in the rest of the paper.

\begin{definition}\label{def:svbeta}
If \(m=0\), set \(\mathbf s(e,\beta)=\mathbf s(\beta)\),
\(M_0=\mathrm{Id}\), and \(H_0=J_0=\varnothing\).
For \(m>0\), fix \(l\in[m]\), and set \(i=i_{p_l}\). We define the mutation sequence
\(\widetilde{\mu}_l\) by
\begin{equation}
\widetilde{\mu}_l=
\begin{cases}
\mu_{(i,n_i-a_l)}\circ \cdots \circ \mu_{(i,b_l+1)},
& \text{if } a_l+b_l<n_i,\\[4pt]
\Id,
& \text{if } a_l+b_l=n_i.
\end{cases}
\end{equation}
Here the composition is applied from right to left. Thus, in the first case,
the mutations are performed successively at
\[
(i,b_l+1),\ (i,b_l+2),\ \dots,\ (i,n_i-a_l).
\]

We further set
\begin{equation}
M_l:=\widetilde{\mu}_l\circ \cdots \circ \widetilde{\mu}_1 .
\end{equation}
We use the same notation for the induced sequence of quiver mutations.

We define a sequence of seeds \(\widetilde{\mathbf{s}}_l\) inductively, starting
from
\[
\widetilde{\mathbf{s}}_0=\mathbf{s}(\beta).
\]
Assume that \(\widetilde{\mathbf{s}}_{l-1}\) has been constructed. We apply the
mutation sequence \(\widetilde{\mu}_l\) to \(\widetilde{\mathbf{s}}_{l-1}\), and
then freeze the vertex
\[
(i_{p_l},\, n_{i_{p_l}}-a_l+1).
\]
The resulting seed is denoted by \(\widetilde{\mathbf{s}}_l\).
No vertices are deleted at an intermediate stage. For later use, write
\[
H_l:=\{(i_{p_s},n_{i_{p_s}}-a_s+1)\mid 1\le s\le l\},
\qquad H_0=\varnothing.
\]
Every vertex in \(H_l\) is frozen and is not mutated again. The full
intermediate seed and the auxiliary support quivers of
Section~\ref{sec:subalgebra} will be kept distinct.

Finally, starting from the seed \(\widetilde{\mathbf{s}}_m\), we delete the
frozen vertices
\[
(i_{p_l},\, n_{i_{p_l}}-a_l+1),
\qquad l\in[m].
\]
We denote this set of deleted vertices by
\[
J_m:=\{(i_{p_l},\, n_{i_{p_l}}-a_l+1)\mid l\in[m]\},
\]
and set
\[
J:=K\setminus J_m.
\]

Let \(B^{(m)}:=M_m(B_\beta)\) and \(L^{(m)}:=M_m(L_\beta)\). We freeze all
remaining vertices which are adjacent, in the quiver \(M_m(Q_\beta)\), to one
of the vertices in \(J_m\). Denote the resulting set of frozen vertices in
\(J\) by \(J_{\rm fr}\), and set
\[
J_{\rm ex}:=J\setminus J_{\rm fr}.
\]

Note that the exchange matrix of \(\widetilde{\mathbf{s}}_m\) satisfies
\[
B^{(m)}_{J_m\times J_{\rm ex}}=0.
\]
Therefore, deleting the vertices in \(J_m\) does not affect the compatibility
condition on the remaining exchangeable columns. Hence the restricted pair
\[
\left(L^{(m)}_{J\times J},\, B^{(m)}_{J\times J_{\rm ex}}\right)
\]
is compatible.

The resulting seed is denoted by
\[
\mathbf{s}(v,\beta)
=
\left(
\{y_i\}_{i\in J},
L^{(m)}_{J\times J},
B^{(m)}_{J\times J_{\rm ex}},
J_{\rm ex}
\right),
\]
where \(y_i\) denotes the cluster variable attached to the vertex \(i\) in
\(\widetilde{\mathbf{s}}_m\).

We denote by
\[
\cA(\mathbf{s}(v,\beta)),\quad
U(\mathbf{s}(v,\beta)),\quad
\cA_t(\mathbf{s}(v,\beta)),\quad
\overline{\cA}_t(\mathbf{s}(v,\beta))
\]
the cluster algebra, the upper cluster algebra, the quantum cluster algebra
with frozen variables inverted, and the quantum cluster algebra without
inverting frozen variables, respectively, associated with
\(\mathbf{s}(v,\beta)\).

We also denote by
\[
\cA_0(\mathbf{s}(v,\beta))
\]
the integral subalgebra generated by all cluster variables, including the
frozen variables, without inverting them. Thus
\(\cA(\mathbf{s}(v,\beta))\) is the rational scalar extension of
\(\cA_0(\mathbf{s}(v,\beta))_{\rm loc}\).
\end{definition}

\begin{example}\label{exm:mutations}
Continuing Example~\ref{exm:words}, we have
\[
n_1=2,\quad n_2=3,\quad n_3=3,
\]
and
\[
a_1=1,\ a_2=1,\ a_3=2,\ a_4=1,\ a_5=2,
\qquad
b_1=0,\ b_2=0,\ b_3=0,\ b_4=1,\ b_5=1.
\]

Then
\[
\widetilde{\mu}_1=\mu_{(3,2)}\mu_{(3,1)}=\mu_5\mu_1,\qquad
\widetilde{\mu}_2=\mu_{(2,2)}\mu_{(2,1)}=\mu_4\mu_2,
\]
\[
\widetilde{\mu}_3=\mu_{(3,1)}=\mu_1,\qquad
\widetilde{\mu}_4=\Id,\qquad
\widetilde{\mu}_5=\Id.
\]

The deleted vertices of the seed $\widetilde{\mathbf{s}}_5$ are given by
\[
\{(3,3),(2,3),(3,2),(1,2),(2,2)\}=\{4,5,6,7,8\}.
\]
Hence, the seed $\mathbf{s}(v,\beta)$ has vertex set $\{1,2,3\}$.
The frozen vertices are those vertices in $\{1,2,3\}$ that are connected to
$\{4,5,6,7,8\}$ in the quiver $M_5(Q_\beta)$.
\end{example}

\section{Twisted products of flag varieties}\label{sec:twisted}

In this section, we recall the notion of twisted products of flag varieties and
their cluster structures.

Let \(G\) be a connected, simply connected, simple algebraic group over
\(\mathbb C\). Fix a maximal torus \(T\subset G\), and let \(B^+\) and \(B^-\)
be a pair of opposite Borel subgroups containing \(T\). Let
\[
\mathcal B=G/B^+
\]
be the flag variety of \(G\). For each \(w\in W\), we choose a representative
\(\dot w\in N_G(T)\).

For \(w\in W\), we denote by
\[
\mathring{\mathcal B}_w:=B^+\dot wB^+/B^+
\]
the Schubert cell, and for \(v\in W\), we denote by
\[
\mathring{\mathcal B}^v:=B^-\dot vB^+/B^+
\]
the opposite Schubert cell. Their intersection
\[
\mathring{\mathcal B}_{w,v}
:=
\mathring{\mathcal B}_w\cap \mathring{\mathcal B}^v
\]
is called the open Richardson cell. It is nonempty if and only if
\(v\le w\) in the Bruhat order.

Let \(U^+\) and \(U^-\) be the unipotent radicals of \(B^+\) and \(B^-\),
respectively. For \(w\in W\), define
\[
N(w):=U^+\cap \dot wU^-\dot w^{-1}.
\]

Let \(V(\varpi_i)\) be the irreducible \(G\)-module of highest weight
\(\varpi_i\). Fix a highest weight vector
\(\eta_i\in V(\varpi_i)\), and let
\(\eta_i^*\in V(\varpi_i)^*_{-\varpi_i}\) be the dual weight vector
normalized by
\[
\langle \eta_i^*,\eta_i\rangle=1.
\]
We define
\[
\Delta_{\varpi_i}(g)
:=
\langle \eta_i^*,g\eta_i\rangle,
\qquad g\in G.
\]
More generally, for \(u,v\in W\), we define the generalized minor by
\[
\Delta_{u\varpi_i,v\varpi_i}(g)
:=
\Delta_{\varpi_i}(\dot u^{-1}g\dot v).
\]



\subsection{Twisted products of flag varieties}

Let
\[
\mathcal Z
:=
G\times^{B^+}G\times^{B^+}\cdots\times^{B^+}G/B^+
\]
be the twisted product with \(n\) factors. Equivalently, \(\mathcal Z\) is the
quotient of \(G^n\) by the right action of \((B^+)^n\) given by
\[
(g_1,\dots,g_n)\cdot(b_1,\dots,b_n)
=
(g_1b_1,b_1^{-1}g_2b_2,\dots,b_{n-1}^{-1}g_nb_n).
\]
We write a point of \(\mathcal Z\) as \([g_1,\dots,g_n]\).

For a word
\(
\overline w=(w_1,\dots,w_n)
\)
with \(w_k\in W\), define
\[
\mathcal Z_{\overline w}
:=
B^+\dot w_1B^+
\times^{B^+}
B^+\dot w_2B^+
\times^{B^+}
\cdots
\times^{B^+}
B^+\dot w_nB^+/B^+.
\]

The multiplication morphism
\(
m:\mathcal Z\to \mathcal B
\)
is defined by
\[
m([g_1,\dots,g_n])
=
g_1\cdots g_nB^+.
\]
This is well-defined with respect to the above \(B^+\)-actions.

For \(v\in W\), set
\[
\mathring{\mathcal Z}^{\,v}
:=
m^{-1}(\mathring{\mathcal B}^{\,v}),
\]
and define
\[
\mathring{\mathcal Z}_{v,\overline w}
:=
\mathring{\mathcal Z}^{\,v}\cap \mathcal Z_{\overline w}.
\]

In particular, for a word \(\beta=(i_1,\dots,i_N)\) in the index set \(I\), we
obtain the word
\(
\overline w_\beta=(s_{i_1},\dots,s_{i_N})
\)
in \(W\). We write
\[
\mathring{\mathcal Z}_{v,\beta}
:=
\mathring{\mathcal Z}_{v,\overline w_\beta}.
\]

\begin{remark}
By \cite{bao2022total}, the variety
\(\mathring{\mathcal Z}_{v,\beta}\) is nonempty if and only if
\(
v\le \delta(\beta).
\)
In this case, it is a smooth affine variety of dimension
\(N-\ell(v).
\)

More generally, for a word \(\overline w=(w_1,\dots,w_n)\) in \(W\), after
choosing reduced expressions for all \(w_k\) and concatenating them, one obtains
a word \(\beta\) in the simple reflections. Then there is an isomorphism
\[
\mathring{\mathcal Z}_{v,\overline w}
\cong
\mathring{\mathcal Z}_{v,\beta}.
\]
\end{remark}

\subsection{Braid varieties}

For a pair of flags \((xB^+,yB^+)\), we say that they are in relative position
\(w\in W\) if
\(x^{-1}y\in B^+\dot wB^+.
\)
We denote this relation by
\[
xB^+\xrightarrow{w} yB^+ .
\]

For each \(i\in I\), fix isomorphisms
\[
x_i:\mathbb C\to U_i^+,
\qquad
y_i:\mathbb C\to U_i^-,
\]
where \(U_i^+\) and \(U_i^-\) are the positive and negative root subgroups
corresponding to \(\alpha_i\). We choose these isomorphisms so that the assignments
\[
\begin{pmatrix}
1 & z \\ 0 & 1
\end{pmatrix}
\mapsto x_i(z),
\qquad
\begin{pmatrix}
b & 0 \\ 0 & b^{-1}
\end{pmatrix}
\mapsto \chi_i(b),
\qquad
\begin{pmatrix}
1 & 0 \\ z & 1
\end{pmatrix}
\mapsto y_i(z)
\]
define a morphism
\[
\varphi_i:\mathrm{SL}_2(\mathbb C)\to G,
\]
where \(\chi_i:\mathbb C^\times\to T\) is the simple coroot corresponding to
\(i\).

Let
\[
\dot s_i:=
\varphi_i
\begin{pmatrix}
0 & -1 \\ 1 & 0
\end{pmatrix}
\in G .
\]
For \(z\in\mathbb C\), define
\begin{equation}\label{eq:Biz}
B_i(z)
:=
\varphi_i
\begin{pmatrix}
z & -1 \\ 1 & 0
\end{pmatrix}
=
x_i(z)\dot s_i .
\end{equation}

\begin{proposition}
\label{pro:Biz}
{\rm \cite[Proposition~3.6]{casals2025cluster}.}
Fix a flag \(xB^+\in\mathcal B\). Then
\[
\{yB^+\in\mathcal B\mid xB^+\xrightarrow{s_i}yB^+\}
=
\{xB_i(z)B^+\mid z\in\mathbb C\}.
\]
Moreover,
\[
xB_i(z)B^+=xB_i(z')B^+
\]
if and only if \(z=z'\).
\end{proposition}

\begin{definition}
Let \(\beta=(i_1,\dots,i_r)\) be a positive braid word, and let
\(\delta(\beta)\in W\) be its Demazure product. The \emph{braid variety}
\(X(\beta)\) is defined by
\[
X(\beta)
:=
\left\{
(z_1,\dots,z_r)\in\mathbb C^r
\ \middle|\
B_{i_1}(z_1)\cdots B_{i_r}(z_r)
\in
\dot{\delta(\beta)}B^+
\right\}.
\]
By \cite[Corollary~3.7]{casals2025cluster}, this definition is equivalent to
the following flag-theoretic description:
\[
X(\beta)
\simeq
\left\{
(F_1,\dots,F_r)\in\mathcal B^r
\ \middle|\
B^+\xrightarrow{s_{i_1}}F_1
\xrightarrow{s_{i_2}}\cdots
\xrightarrow{s_{i_r}}F_r,\ 
F_r=\dot{\delta(\beta)}B^+
\right\}.
\]
\end{definition}

\begin{proposition}\label{pro:braid}
Let \(b\in \Br^+\), and let \(\beta=(i_1,\dots,i_r)\) be an expression of \(b\).
Then there exists an isomorphism
\[
\mathring{\mathcal{Z}}_{\delta(b),\beta}\cong X(\beta).
\]
\end{proposition}

\begin{proof}
For each \(k\), the Schubert cell corresponding to \(s_{i_k}\) is parametrized by
\[
B^+\dot{s}_{i_k}B^+/B^+
=
\{x_{i_k}(z)\dot{s}_{i_k}B^+\mid z\in\mathbb C\}
=
\{B_{i_k}(z)B^+\mid z\in\mathbb C\}.
\]
Therefore the map
\[
\mathbb C^r\longrightarrow \mathcal Z_\beta,
\qquad
(z_1,\dots,z_r)
\longmapsto
[B_{i_1}(z_1),\dots,B_{i_r}(z_r)]
\]
is an isomorphism.

Under this parametrization, the multiplication map
\(
m:\mathcal Z_\beta\to \mathcal B
\)
is given by
\[
m([B_{i_1}(z_1),\dots,B_{i_r}(z_r)])
=
B_{i_1}(z_1)\cdots B_{i_r}(z_r)B^+.
\]
Hence
\[
\mathring{\mathcal Z}_{\delta(b),\beta}
\cong
\left\{
(z_1,\dots,z_r)\in\mathbb C^r
\ \middle|\
B_{i_1}(z_1)\cdots B_{i_r}(z_r)B^+
\in
B^-\dot{\delta(b)}B^+/B^+
\right\}.
\]

On the other hand, one has,
\[
B^+\dot{s}_{i_1}B^+\cdots B^+\dot{s}_{i_r}B^+/B^+
\subset
\overline{B^+\dot{\delta(b)}B^+/B^+}.
\]
Therefore every point in the image of \(m\) lies in the Schubert variety
\(
\overline{B^+\dot{\delta(b)}B^+/B^+}.
\)
Since
\[
\overline{B^+\dot{\delta(b)}B^+/B^+}
\cap
B^-\dot{\delta(b)}B^+/B^+
=
\{\dot{\delta(b)}B^+\},
\]
the above condition is equivalent to
\[
B_{i_1}(z_1)\cdots B_{i_r}(z_r)B^+
=
\dot{\delta(b)}B^+.
\]
Equivalently,
\(
B_{i_1}(z_1)\cdots B_{i_r}(z_r)
\in
\dot{\delta(b)}B^+.
\)
Thus
\[
\mathring{\mathcal Z}_{\delta(b),\beta}
\cong
\left\{
(z_1,\dots,z_r)\in\mathbb C^r
\ \middle|\
B_{i_1}(z_1)\cdots B_{i_r}(z_r)
\in
\dot{\delta(b)}B^+
\right\}
=
X(\beta).
\]
This proves the claim.
\end{proof}
In the Dynkin type, every twisted product of flag varieties is a braid variety.

\begin{lemma}\label{lem:isoZvbeta}
Let \(\beta=(i_1,\dots,i_r)\) be a word for \(b\in \Br^+\), and let
\(v\leq \delta(b)\). Set
\(
v^c:=v^{-1}w_0.
\)
Let
\(
\overline{v^c}=(i_{r+1},\dots,i_{r+q})
\)
be a reduced expression of \(v^c\), where \(q=\ell(v^c)\). Then there is an
isomorphism
\[
\mathring{\mathcal{Z}}_{v,\beta}\cong X(\beta\,\overline{v^c}).
\]
\end{lemma}

\begin{proof}
Since \(v\le \delta(b)=\delta(\beta)\), we have
\[
\delta(\beta\,\overline{v^c})
=
\delta(\beta)* v^{-1}w_0
=
w_0,
\]
where \(*\) denotes the Demazure product.

Let
\begin{equation}\label{eq:B1sequence}
B^+
\xrightarrow{s_{i_1}} B_1
\xrightarrow{s_{i_2}} \cdots
\xrightarrow{s_{i_r}} B_r
\xrightarrow{s_{i_{r+1}}} B_{r+1}
\xrightarrow{s_{i_{r+2}}} \cdots
\xrightarrow{s_{i_{r+q}}} B_{r+q}
=
\dot w_0B^+
\end{equation}
be an element of \(X(\beta\,\overline{v^c})\).

Since \(\overline{v^c}\) is a reduced expression of \(v^c=v^{-1}w_0\), the last
part of the sequence implies that
\[
B_r\xrightarrow{v^{-1}w_0}\dot w_0B^+.
\]
Write \(B_r=gB^+\). Then
\[
g^{-1}\dot w_0\in B^+\dot v^{-1}\dot w_0B^+.
\]
Therefore
\[
g\in \dot w_0B^+\dot w_0^{-1}\dot vB^+
=
B^-\dot vB^+.
\]
Hence
\[
B_r\in B^-\dot vB^+/B^+
=
\mathring{\mathcal B}^{\,v}.
\]

By Proposition~\ref{pro:Biz}, the first \(r\) arrows in
\eqref{eq:B1sequence} are uniquely parametrized by elements
\(z_1,\dots,z_r\in\mathbb C\), namely
\[
B_k=
B_{i_1}(z_1)\cdots B_{i_k}(z_k)B^+,
\qquad 1\le k\le r.
\]
Thus the first part of the sequence determines a point of
\(\mathring{\mathcal Z}_{v,\beta}\).

Conversely, take a point of \(\mathring{\mathcal Z}_{v,\beta}\). In the
coordinates of \(\mathcal Z_\beta\), it is represented by
\[
(z_1,\dots,z_r)\in\mathbb C^r
\]
such that
\[
B_r:=
B_{i_1}(z_1)\cdots B_{i_r}(z_r)B^+
\in
B^-\dot vB^+/B^+.
\]
Equivalently,
\[
B_r\xrightarrow{v^{-1}w_0}\dot w_0B^+.
\]
Since \(\overline{v^c}\) is a reduced expression of \(v^{-1}w_0\), by
\cite[Lemma~3.2]{casals2025cluster}, the flag \(B_r\) uniquely extends to a
sequence
\[
B_r
\xrightarrow{s_{i_{r+1}}} B_{r+1}
\xrightarrow{s_{i_{r+2}}} \cdots
\xrightarrow{s_{i_{r+q}}} B_{r+q}
=
\dot w_0B^+.
\]
Together with the first \(r\) flags determined by \(z_1,\dots,z_r\), this gives
an element of \(X(\beta\,\overline{v^c})\).

The two constructions are inverse to each other, and they are regular in the
coordinates above. Hence they define an isomorphism of varieties.
\end{proof}

\subsection{Double Bruhat cells}

For \(v,w\in W\), the \emph{double Bruhat cell} is defined by
\[
G^{v,w}:=B^+\dot wB^+\cap B^-\dot vB^-.
\]
Let
\[
\pi:G\to G/T=:L
\]
be the quotient map for the right action of \(T\). The \emph{reduced double
Bruhat cell} is defined as
\[
L_{v,w}:=\pi(G^{v,w}).
\]

In \cite{bao2025upper}, Bao and Ye proved that the coordinate rings of twisted
products of flag varieties carry upper cluster algebra structures. We also
recall the following realization of reduced double Bruhat cells as special
cases of twisted products of flag varieties.

\begin{proposition}[{\cite[Proposition~2.1]{webster2007deodhar}}]
There is an isomorphism
\[
L_{v,w}\cong \mathring{\mathcal{Z}}_{v w_0,(w,w_0)}.
\]
\end{proposition}

\subsection{Cluster structure on twisted products of flag varieties}

In this section, we recall the cluster structure on twisted products of flag
varieties.

\begin{theorem}[{\cite[Theorem~6.5 and Remark~6.1]{bao2025upper}}]
Let \(\beta\) be a word and let \(v\le \delta(\beta)\). Then the coordinate
ring of \(\mathring{\mathcal{Z}}_{v,\beta}\) admits an upper cluster algebra
structure given by the seed \(\mathbf{s}(v,\beta)\). More precisely, there is
an isomorphism
\[
U(\mathbf{s}(v,\beta))\otimes_{\mathbb Q}\mathbb C
\cong
\mathbb C[\mathring{\mathcal{Z}}_{v,\beta}].
\]
\end{theorem}

For \(\beta=(i_1,\ldots,i_r)\), we use the notation
\[
\operatorname{Conf}(\beta)
:=
\left\{(z_1,\ldots,z_r)\in\mathbb C^r\ \middle|\,
B_{i_1}(z_1)\cdots B_{i_r}(z_r)\in B^-B^+\right\}.
\]
Thus \(\operatorname{Conf}(\beta)=\mathring{\mathcal Z}_{e,\beta}\)
in the above coordinates. This is the normalized configuration space
of \cite[Section~3.7]{casals2025cluster}; the initial boundary torus
parameters of the decorated configuration space are fixed to \(1\).
By \cite[Corollary~5.21]{casals2025cluster}, one has
\[
\mathcal A(\mathbf{s}(\beta))\otimes_{\mathbb Q}\mathbb C
\cong
\mathbb C[\operatorname{Conf}(\beta)].
\]

Let \(\beta=(i_1,\dots,i_r)\) be a word in \(I\), and let
\[
\beta^{\operatorname{rev}}=(i_r,\dots,i_1).
\]
Since the Demazure product is compatible with the anti-involution
\(w\mapsto w^{-1}\), we have
\[
\delta(\beta^{\operatorname{rev}})
=
\delta(\beta)^{-1}.
\]

\begin{lemma}\label{lem:betarev}
There is an isomorphism
\[
X(\beta)\cong X(\beta^{\operatorname{rev}}).
\]
\end{lemma}

\begin{proof}
We use the flag-theoretic description of braid varieties. An element of
\(X(\beta)\) is a sequence of flags
\[
F_0=B^+
\xrightarrow{s_{i_1}}
F_1
\xrightarrow{s_{i_2}}
\cdots
\xrightarrow{s_{i_r}}
F_r
=
\dot{\delta(\beta)}B^+.
\]

Define a new sequence of flags by
\[
G_k:=\dot{\delta(\beta)}^{-1}F_{r-k},
\qquad 0\le k\le r.
\]
Then
\[
G_0
=
\dot{\delta(\beta)}^{-1}F_r
=
B^+,
\]
and
\[
G_r
=
\dot{\delta(\beta)}^{-1}F_0
=
\dot{\delta(\beta)}^{-1}B^+
=
\dot{\delta(\beta^{\operatorname{rev}})}B^+.
\]

Moreover, if
\[
F_{j-1}\xrightarrow{s_{i_j}}F_j,
\]
then also
\[
F_j\xrightarrow{s_{i_j}}F_{j-1},
\]
because the double coset \(B^+\dot{s}_{i_j}B^+\) is stable under inverse.
Since relative position is preserved by the left action of \(G\), we get
\[
G_{k-1}
\xrightarrow{s_{i_{r-k+1}}}
G_k
\]
for every \(1\le k\le r\). Hence
\[
G_0=B^+
\xrightarrow{s_{i_r}}
G_1
\xrightarrow{s_{i_{r-1}}}
\cdots
\xrightarrow{s_{i_1}}
G_r
=
\dot{\delta(\beta^{\operatorname{rev}})}B^+.
\]
Thus \((G_1,\dots,G_r)\) defines an element of \(X(\beta^{\operatorname{rev}})\).

This construction is regular, since it is induced by reversing the sequence and
applying the left action of the fixed element \(\dot{\delta(\beta)}^{-1}\).
Its inverse is explicitly
\[
F_k=\dot{\delta(\beta)}G_{r-k},\qquad 0\leq k\leq r.
\]
This formula is independent of any compatibility between the chosen
representatives of \(\delta(\beta)\) and \(\delta(\beta)^{-1}\).
Therefore the construction is an isomorphism.
\end{proof}

\begin{lemma}\label{lem:zv-1beta}
Let \(\beta=(i_1,\dots,i_r)\) be a word in \(I\), and let
\(v\leq\delta(\beta)\). Then there is an isomorphism
\[
\mathring{\mathcal{Z}}_{v,\beta}
\cong
\mathring{\mathcal{Z}}_{v^{-1},\beta^{\operatorname{rev}}}.
\]
\end{lemma}

\begin{proof}
Let \(\tau:G\to G\) be the Chevalley anti-involution determined by
the pinning, so that
\[
\tau(x_i(z))=y_i(z),\qquad
\tau(y_i(z))=x_i(z),\qquad
\tau(t)=t\quad(t\in T).
\]
It exchanges \(B^+\) and \(B^-\), and induces inversion on \(W\).
Consequently,
\begin{equation}\label{eq:transpose-Bruhat-cell}
\tau(B^-\dot vB^+)=B^-\dot{v^{-1}}B^+.
\end{equation}
For \(t_i:=\chi_i(-1)\), a calculation in \(\mathrm{SL}_2\) gives
\begin{equation}\label{eq:transpose-Bi}
\tau(B_i(z))=B_i(-z)t_i.
\end{equation}
We also have, for \(t\in T\),
\begin{equation}\label{eq:torus-through-Bi}
tB_i(z)=B_i\bigl(\alpha_i(t)z\bigr)s_i(t),
\end{equation}
where \(s_i(t)=\dot s_i^{-1}t\dot s_i\).

Write \(j_k=i_{r+1-k}\). Define fixed elements \(h_k\in T\)
and constants \(\varepsilon_k\) by
\[
h_0=1,\qquad
h_k=s_{j_k}(h_{k-1})t_{j_k},\qquad
\varepsilon_k=-\alpha_{j_k}(h_{k-1})\quad(1\leq k\leq r).
\]
Each \(h_k\) has order dividing \(2\), so
\(\varepsilon_k\in\{1,-1\}\). Applying
\eqref{eq:transpose-Bi} and \eqref{eq:torus-through-Bi} successively,
we obtain
\begin{equation}\label{eq:transpose-word}
\tau\bigl(B_{i_1}(z_1)\cdots B_{i_r}(z_r)\bigr)
=
B_{j_1}(\varepsilon_1z_r)\cdots
B_{j_r}(\varepsilon_rz_1)h_r.
\end{equation}
Right multiplication by \(h_r\in T\) preserves every double coset
\(B^-\dot uB^+\). Therefore
\eqref{eq:transpose-Bruhat-cell} and \eqref{eq:transpose-word} show
that the linear automorphism
\[
(z_1,\ldots,z_r)\longmapsto
(\varepsilon_1z_r,\ldots,\varepsilon_rz_1)
\]
of \(\mathbb C^r\) carries the locus
\(B_{i_1}(z_1)\cdots B_{i_r}(z_r)\in B^-\dot vB^+\)
onto the corresponding locus for
\((v^{-1},\beta^{\operatorname{rev}})\).
The parametrizations used in Proposition~\ref{pro:braid} now give
the asserted isomorphism, including its regular inverse.
\end{proof}

\begin{lemma}[Restriction after a terminal left insertion]
\label{lem:terminal-left-restriction}
Assume that \(G\) is of simply-laced Dynkin type. Let \(\eta\) be a
positive braid word, put \(d=\delta(\eta)\), and suppose that
\(\delta(i\eta)=d\). Let \(\mathfrak w\) be a double inductive weave
for \(\eta\), and let \(\mathfrak w^+\) be obtained by appending the
entry \(iL\) to its double string. Denote its new trivalent vertex by
\(c\), and write \(A_u\) and \(A_u^+\) for the cluster variables of
\(\mathfrak w\) and \(\mathfrak w^+\), respectively.

For \(z\in\mathbb C^\times\), put \(U_i(z)=x_i(-z^{-1})\).
Let \(\Phi_{i,z}\) be the coordinate change obtained by moving
\(U_i(z)^{-1}\) successively through the factors of \(B_\eta(a)\),
using \cite[Corollary~3.9 and Lemma~4.2]{casals2025cluster}.
It satisfies
\begin{equation}\label{eq:terminal-left-coordinate-change}
U_i(z)B_\eta\bigl(\Phi_{i,z}(a)\bigr)
=B_\eta(a)b(a,z),\qquad b(a,z)\in B^+,
\end{equation}
where \(B_\eta(a)\) is the ordered product of braid matrices for
\(\eta\). By \cite[Lemmas~4.2 and~5.24\textup{(2)}]{casals2025cluster},
\[
\begin{aligned}
\Psi_i:X(\eta)\times\mathbb C^\times
&\xrightarrow{\sim}\{z\ne0\}\subset X(i\eta),\\
(a,z)&\longmapsto\bigl(z,\Phi_{i,z}(a)\bigr).
\end{aligned}
\]
Define \(\iota_i(a)=\Psi_i(a,1)\). Then
\begin{equation}\label{eq:terminal-left-variable-restriction}
\iota_i^*(A_u^+)=A_u\qquad(u\ne c).
\end{equation}
The ordinary seed of \(\mathfrak w\) is obtained from that of
\(\mathfrak w^+\) by deleting the frozen vertex \(c\), freezing its
remaining neighbors, and using these restricted variables.
As usual, arrows between frozen vertices are omitted.
\end{lemma}

\begin{proof}
Choose the bottom reduced word of \(\mathfrak w\) to begin with \(i\),
using braid equivalences if necessary. This is possible because
\(s_i d<d\). The terminal insertion adds an outer \(i\)-strand and
merges it with this first bottom strand at \(c\). Its new cycle runs
straight to the bottom, so \(c\) is frozen. Every old cycle has weight
zero on the northwest edge at \(c\). If its northeast weight is
\(w_u\), its outgoing weight is therefore \(\min(0,w_u)=0\).
Consequently, an old vertex that is mutable in \(\mathfrak w^+\)
becomes frozen upon removal of \(c\) exactly when \(w_u>0\).
For such a mutable vertex the boundary intersection term vanishes,
and its intersection with the new cycle has absolute value \(w_u\).
Thus these are exactly the mutable neighbors of \(c\).
For any column remaining mutable after removal, the corresponding
cycle has zero bottom weights already in \(\mathfrak w\); the new
trivalent vertex contributes no intersection with an old cycle in
that column, and both boundary terms vanish. Hence all surviving
exchangeable columns agree. This is the local computation of
\cite[Definitions~4.17,~4.26 and Lemma~4.47]{casals2025cluster}; it
uses no assumption on the earlier left or right insertions.

Set
\[
L_i(z):=
\varphi_i\!\begin{pmatrix}z&0\\1&z^{-1}\end{pmatrix}
=y_i(z^{-1})\chi_i(z),
\]
so that \(B_i(z)=L_i(z)U_i(z)\) for \(z\neq0\). Since
\(s_i d<d\), the length criterion gives
\(d^{-1}(\alpha_i)<0\). Hence
\[
c_i(z):=\dot d^{-1}L_i(z)\dot d\in B^+,
\]
because conjugation sends the root subgroup
\(U_{-\alpha_i}\) to the positive root subgroup
\(U_{-d^{-1}(\alpha_i)}\), while preserving \(T\).

Let \(a'=\Phi_{i,z}(a)\). By
\eqref{eq:terminal-left-coordinate-change},
\[
U_i(z)B_\eta(a')=B_\eta(a)b(a,z),
\qquad b(a,z)\in B^+.
\]
Therefore
\[
\dot d^{-1}B_i(z)B_\eta(a')
=
c_i(z)\,\dot d^{-1}B_\eta(a)\,b(a,z).
\]
Thus the left-hand side belongs to \(B^+\) if and only if
\(\dot d^{-1}B_\eta(a)\in B^+\). Since the coordinate change
\(\Phi_{i,z}\) is reversible, this proves that
\[
\Psi_i(a,z)=\bigl(z,\Phi_{i,z}(a)\bigr)
\]
is the stated isomorphism.

It remains to compare the cluster variables. By
\cite[Lemma~4.2]{casals2025cluster}, the coordinate change induced
by \(U_i(z)\) multiplies each right incoming edge coordinate by a
scalar determined by the \(T\)-component of \(U_i(z)\). Since
\(U_i(z)\in U^+\) has trivial \(T\)-component, all such coordinates
at the old trivalent vertices remain unchanged. Moreover,
\cite[Equation~(25)]{casals2025cluster} expresses \(A_u\)
recursively in terms of its right incoming edge coordinate and the
cluster variables associated with vertices above \(u\). A downward
induction therefore gives
\[
\Psi_i^*(A_u^+)=\operatorname{pr}_1^*(A_u)
\qquad (u\neq c).
\]
Finally, since \(\iota_i(a)=\Psi_i(a,1)\), we obtain
\[
\iota_i^*(A_u^+)=A_u,
\]
which proves
\eqref{eq:terminal-left-variable-restriction}.
These equalities on the weave tori extend to the regular functions.
At \(c\), the same formula gives
\[
z=(A_c^+)^{-1}\prod_{u\ne c}(A_u^+)^{w_u},
\qquad
\iota_i^*(A_c^+)=\prod_{u\ne c}A_u^{w_u},
\]
as in \cite[Lemma~5.26 and Equation~(32)]{casals2025cluster}.
Thus the normalization fixes the crossing coordinate \(z\) to \(1\);
the deleted cluster variable restricts to the displayed monomial.
Every factor with \(w_u>0\) is frozen for \(\mathfrak w\), since its
cycle reaches its bottom boundary, so this monomial is invertible.
\end{proof}

\begin{lemma}\label{lem:inductivewave}
Assume that \(G\) is of simply-laced Dynkin type. After identifying
the surviving vertex labels, the ordinary cluster seed induced by the
right-inductive weave
\[
\overrightarrow{\mathfrak m}(\overline{w_0v^{-1}}\beta)
\]
is obtained from \(\mathbf{s}(\beta)\) by the mutation, freezing, and deletion
procedure of Definition~\ref{def:svbeta}. In particular, it is isomorphic
to the ordinary cluster seed underlying \(\mathbf{s}(v,\beta)\).
\end{lemma}

\begin{proof}
We use double inductive weaves and the elementary interchanges of
\cite[Section~6.4]{casals2025cluster}. An entry \(iR\), respectively
\(iL\), inserts \(i\) on the right, respectively the left. We call
an entry \emph{special} if it increases the length of the Demazure
product, and \emph{solid} otherwise; solid entries index the
trivalent vertices and the cluster variables. We give the Demazure
calculation explicitly, since \(\beta\) need not be reduced.

Retain \(m=\ell(v)\), and put
\[
k_l=i_{p_l},\qquad
u_l=s_{k_1}\cdots s_{k_l}\quad(0\leq l\leq m),\qquad u_0=e.
\]
Choose a reduced word \(\gamma=(j_1,\ldots,j_q)\) for
\(w_0v^{-1}\), where \(q=\ell(w_0)-m\). Then
\[
\Delta=(k_1^*,\ldots,k_m^*,j_1,\ldots,j_q)
\]
is a reduced word for \(w_0\). For \(0\leq l\leq m\), set
\[
\Gamma_l=(k_{l+1}^*,\ldots,k_m^*,j_1,\ldots,j_q),\qquad
h_l=w_0u_l^{-1}.
\]
The word \(\Gamma_l\) is reduced and represents \(h_l\).
For \(0\leq a\leq r\), let
\(t_a=\#\{s\in[m]\mid p_s\leq a\}\).
We claim that
\begin{equation}\label{eq:weave-comparison-Demazure}
\delta\bigl(\Gamma_l(i_1,\ldots,i_a)\bigr)
=w_0u_l^{-1}u_{\min\{l,t_a\}}.
\end{equation}

To prove the claim, observe first that
\((i_{p_1},\ldots,i_{p_l})\) is the leftmost reduced
subexpression for \(u_l\) in \((i_1,\ldots,i_{p_l})\).
An earlier such subexpression could be concatenated with
\((i_{p_{l+1}},\ldots,i_{p_m})\), contradicting the leftmost
choice for \(v\). By Lemma~\ref{lem:vk}, the greedy residuals for
\(u_l\) along \(\beta\) are therefore
\[
z_{a,l}=u_{\min\{l,t_a\}}^{-1}u_l.
\]
They satisfy \(z_{0,l}=u_l\) and
\(z_{a,l}=\min\{z_{a-1,l},s_{i_a}z_{a-1,l}\}\);
after \(a=p_l\), the residual is \(e\).
Since \(x\mapsto w_0x^{-1}\) reverses Bruhat order,
\[
(w_0z_{a-1,l}^{-1})* s_{i_a}=w_0z_{a,l}^{-1}.
\]
Induction on \(a\) proves
\eqref{eq:weave-comparison-Demazure}. In particular,
\begin{equation}\label{eq:weave-comparison-length}
\ell\!\left(\delta\bigl(\Gamma_l(i_1,\ldots,i_a)\bigr)\right)
=\ell(w_0)-l+\min\{l,t_a\}.
\end{equation}

Start with the all-right double string for \(\Delta\beta\).
Under \(X(\Delta\beta)\cong\operatorname{Conf}(\beta)\), its ordinary
seed is identified with that underlying \(\mathbf s(\beta)\) by
\[
X_k\longmapsto
\Delta_{\varpi_{i_k}}\bigl(B_{i_1}(z_1)\cdots B_{i_k}(z_k)\bigr),
\qquad 1\leq k\leq r.
\]
Here \cite[Lemma~3.16, Corollary~4.45 and Proposition~5.20]{casals2025cluster}
identify the coordinates, the quiver, and the cluster variables,
respectively; the frozen vertices are the last occurrences of each color.
At stage \(l\), change the first entry \(k_l^*R\) to \(k_l^*L\)
and move it past \(\Gamma_l\) and then past all the entries of \(\beta\).
The first-entry change does not alter the weave. Moving past \(\Gamma_l\)
uses only special--special interchanges, since the remaining word
for the longest element is reduced. After completing stage \(l\),
the double string is
\begin{equation}\label{eq:comparison-double-string-l}
(\Gamma_lR,i_1R,\ldots,i_rR,k_l^*L,\ldots,k_1^*L),
\end{equation}
where \(\Gamma_lR\) means that every letter of \(\Gamma_l\) is inserted on
the right.

We now inspect the crossing with \(i_aR\). Before \(a=p_l\),
\eqref{eq:weave-comparison-length} shows that the moving
\(k_l^*L\)-entry is special. When \(a\) is one of
\(p_1,\ldots,p_{l-1}\), the interchange is Case~1 in the proof
of \cite[Theorem~6.8]{casals2025cluster}; at the other positions
before \(p_l\), it is Case~3. These interchanges are weave
equivalences and preserve the corresponding cluster variables.
At \(a=p_l\), the relevant products are
\[
\delta\bigl(\Gamma_l(i_1,\ldots,i_{p_l-1})\bigr)=w_0s_{k_l},
\qquad
s_{k_l^*}(w_0s_{k_l})=(w_0s_{k_l})s_{k_l}=w_0.
\]
Thus Case~2 applies: no mutation occurs, and the variable previously
attached to \(i_{p_l}R\) is transferred to the moving
\(k_l^*L\)-entry.

For \(a>p_l\), the Demazure product of the preceding block is
\(w_0\), by \eqref{eq:weave-comparison-Demazure}. Both entries
are solid. Cases~4 and~5 in the same proof show that the interchange
is a mutation exactly when
\[
s_{k_l^*}w_0=w_0s_{i_a},\quad\text{equivalently }i_a=k_l.
\]
All other interchanges are equivalences. This identifies every
mutation in stage \(l\), without imposing reducedness on \(\beta\).

Write \(i=k_l\). Since \(p_l\) is the \((a_l+b_l)\)-th occurrence of \(i\) in \(\beta\), and the \(a_l-1\) earlier selected \(i\)-occurrences have already been moved to the final left block, the variable transferred at \(p_l\) has label \((i,b_l+1)\). Under each subsequent Case~5 in the proof
of \cite[Theorem~6.8]{casals2025cluster} interchange, if the moving entry has label \((i,c)\), mutation is performed at \((i,c)\), after which the new cluster variable carried by the moving entry is relabeled \((i,c+1)\). Thus its label moves one place to the right after each mutation. Since there are \(n_i-a_l-b_l\) later occurrences of \(i\), the successive mutation labels are
$$
(i,b_l+1),\ (i,b_l+2),\ldots,(i,n_i-a_l),
$$
and the final label of the moving entry is
$
\kappa_l=(i,n_i-a_l+1).
$
 This is precisely
\(\widetilde\mu_l\) in Definition~\ref{def:svbeta}, including
the empty-sequence case. After all \(m\) stages the induced
mutation sequence is \(M_m\), and the final left entries have
labels \(\kappa_m,\ldots,\kappa_1\).

We have \(\delta(\gamma\beta)=w_0\), so all entries in the final
left block are solid. Remove them in reverse construction order,
with labels \(\kappa_1,\ldots,\kappa_m\), applying
Lemma~\ref{lem:terminal-left-restriction} at each step. Each removal
deletes a frozen vertex, freezes its remaining neighbors, and
preserves every surviving variable under the normalized restriction.

Thus the deleted labels are \(J_m=\{\kappa_l\mid l\in[m]\}\).
Among the surviving labels, the frozen ones are the original
frozen vertices together with the neighbors of \(J_m\) in
\(M_m(Q_\beta)\). This is exactly the freezing and deletion in
Definition~\ref{def:svbeta}. The quiver left after removal is the
one associated with the all-right string for \(\gamma\beta\).
Together with the variable identifications along weave equivalences
and mutations, the composite restriction therefore identifies the
ordinary seeds, including their cluster variables and frozen vertices.
\end{proof}

\begin{theorem}\label{thm:clusterproduct}
Assume that \(G\) is of simply-laced Dynkin type. Let \(\beta\) be a word in
\(I\), and let \(v\leq \delta(\beta)\). There is an isomorphism of
affine varieties
\[
\mathring{\mathcal Z}_{v,\beta}
\cong
X(\overline{w_0v^{-1}}\beta).
\]
The ordinary seed underlying \(\mathbf{s}(v,\beta)\) defines a
locally acyclic cluster structure on
\(\mathbb C[\mathring{\mathcal Z}_{v,\beta}]\). More precisely,
\[
U(\mathbf{s}(v,\beta))=\mathcal A(\mathbf{s}(v,\beta)),
\qquad
\mathcal A(\mathbf{s}(v,\beta))\otimes_{\mathbb Q}\mathbb C
\cong\mathbb C[\mathring{\mathcal Z}_{v,\beta}].
\]

If \(v=\delta(\beta)\), the ordinary seed underlying
\(\mathbf{s}(\delta(\beta),\beta)\) is isomorphic to the seed
associated with the right-inductive weave
\(
\overrightarrow{\mathfrak m}(\beta).
\)
\end{theorem}

\begin{proof}
First, we identify the relevant twisted product with a braid variety. By
Lemma~\ref{lem:zv-1beta}, Lemma~\ref{lem:isoZvbeta}, and
Lemma~\ref{lem:betarev}, we have a chain of isomorphisms
\[
\mathring{\mathcal Z}_{v,\beta}
\cong
\mathring{\mathcal Z}_{v^{-1},\beta^{\operatorname{rev}}}
\cong
X(\beta^{\operatorname{rev}}\overline{vw_0})
\cong
X(\overline{w_0v^{-1}}\beta).
\]
Here the second isomorphism uses
\[
(v^{-1})^c=vw_0,
\]
and the last isomorphism is the reversal isomorphism for braid varieties.

By Lemma~\ref{lem:inductivewave}, the ordinary seed induced by the
right-inductive weave
\[
\overrightarrow{\mathfrak m}(\overline{w_0v^{-1}}\beta)
\]
is obtained from \(\mathbf{s}(\beta)\) by the mutation, freezing, and deletion
procedure of Definition~\ref{def:svbeta}. The corresponding braid-variety
cluster algebra is locally acyclic by
\cite[Theorem~7.13]{casals2025cluster}. The seed isomorphism therefore
implies the same property for \(\mathcal A(\mathbf{s}(v,\beta))\),
and gives
\[
U(\mathbf{s}(v,\beta))=\mathcal A(\mathbf{s}(v,\beta)).
\]
The localizations and exchange relations in the cited local acyclicity
argument are defined over \(\mathbb Q\), so this equality holds
over the coefficient field fixed in Section~\ref{sec:pre}.
Combining it with the coordinate-ring realization of
\cite[Theorem~6.5]{bao2025upper} proves the claimed cluster
structure on \(\mathring{\mathcal Z}_{v,\beta}\).

Finally, suppose \(v=\delta(\beta)\), and write
\(\gamma=(j_1,\ldots,j_q)\) for the chosen reduced word of
\(h=w_0v^{-1}\). For each prefix \(\beta_{\leq a}\), put
\(w_a=\delta(\beta_{\leq a})\). Successive right Demazure
multiplications show that \(w_a\) lies below \(v\) in right
weak order. Since \(\ell(hv)=\ell(h)+\ell(v)\), it follows that
\[
\ell(hw_a)=\ell(h)+\ell(w_a).
\]
The same length-additivity holds with \(h\) replaced by any
suffix of the reduced word \(\gamma\). Move its letters, starting
with \(j_1\), from the initial right block to terminal left
insertions. At each interchange the moving left entry remains
special, while the right entry retains its original special or
solid status. Thus only Cases~1 and~3 of
\cite[Theorem~6.8]{casals2025cluster} occur. These are equivalences,
and the resulting double string is
\[
(i_1R,\ldots,i_rR,j_qL,\ldots,j_1L).
\]
All the final left insertions increase Demazure length.
Removing them preserves the cycles, cluster variables, and frozen
vertices by the terminal-insertion calculation of
\cite[Lemma~4.47 and Lemma~5.25]{casals2025cluster}.
The remaining weave is \(\overrightarrow{\mathfrak m}(\beta)\).
Together with Lemma~\ref{lem:inductivewave}, this proves the final
seed isomorphism.
\end{proof}

\section{Bosonic extension algebra}\label{sec:bosonic}

Retain the simply-laced finite-type Cartan matrix $C=(c_{ij})_{i,j\in I}$,
and let $q$ be a formal parameter. We use the notation
\[
[n]_q=\frac{q^{n}-q^{-n}}{q-q^{-1}},
\qquad
[n]_q!:=\prod_{k=1}^{n}[k]_q,
\qquad
\qbinom{m}{n}
:=
\frac{[m]_q!}{[n]_q![m-n]_q!}.
\]

\begin{definition}
The \emph{bosonic extension algebra} \(\widehat{\mathcal A}\) is the
\(\mathbb Q(q^{1/2})\)-algebra generated by elements
\[
\{f_{i,k}\mid i\in I,\ k\in\mathbb Z\},
\]
subject to the following relations:
\begin{align}
&\sum_{a=0}^{1-c_{ij}}(-1)^a
\qbinom{1-c_{ij}}{a}
f_{i,p}^{\,1-c_{ij}-a} f_{j,p} f_{i,p}^{\,a}
=0,
\qquad i\neq j,\ p\in\mathbb Z,
\label{eq:Serre}
\\[0.3em]
&f_{i,m}f_{j,p}
=
q^{(-1)^{p-m+1}c_{ij}} f_{j,p}f_{i,m}
+
\delta_{(i,m+1),(j,p)}(1-q^{2}),
\qquad m<p.
\label{eq:bosonic}
\end{align}
Here \(\delta_{x,y}\) denotes the Kronecker delta.
\end{definition}

The relations are homogeneous for
$\operatorname{wt}(f_{i,m})=(-1)^m\alpha_i$, giving the decomposition
$\widehat{\mathcal A}=\bigoplus_{\alpha\in Q}\widehat{\mathcal A}_\alpha$.
For $-\infty\leq a\leq b\leq+\infty$, let
$\widehat{\mathcal A}[a,b]$ be the subalgebra generated by the
$f_{i,k}$ with $a\leq k\leq b$. Write
\[
    \widehat{\mathcal A}[m]:=\widehat{\mathcal A}[m,m],\qquad
    \widehat{\mathcal A}_{\geq0}:=\widehat{\mathcal A}[0,+\infty],\qquad
    \widehat{\mathcal A}_{<0}:=\widehat{\mathcal A}[-\infty,-1].
\]
For finite $[a,b]$, ordered multiplication gives the vector-space
factorization
\[
    \widehat{\mathcal A}[b]\otimes\cdots\otimes\widehat{\mathcal A}[a]
    \xrightarrow{\ \sim\ }\widehat{\mathcal A}[a,b].
\]
Passing to the direct limit yields
$\widehat{\mathcal A}=\widehat{\mathcal A}_{\geq0}\cdot
\widehat{\mathcal A}_{<0}$, with larger level indices on the left.

The bar involution and the shift $\mathcal D$ are the
$\mathbb Q$-algebra anti-automorphisms determined by
\[
\begin{aligned}
    \overline{f_{i,k}}&=f_{i,k},&
    \overline{q^{\pm1/2}}&=q^{\mp1/2},\\
    \mathcal D(f_{i,k})&=f_{i,k+1},&
    \mathcal D(q^{\pm1/2})&=q^{\mp1/2}.
\end{aligned}
\]
For $x\in\widehat{\mathcal A}_\alpha$, set
$c(x):=q^{(\alpha,\alpha)/2}\overline{x}$, and extend $c$ additively
across the weight decomposition. Thus
$c(ax)=\overline a\,c(x)$ for $a\in\mathbb Q(q^{1/2})$.

\subsection{Braid symmetries and global basis}

\begin{proposition}[{\cite[Theorem~3.1]{kashiwara2024braid}}]\label{pro:braidsym}
For each \(i\in I\), there exist \(\mathbb Q(q^{1/2})\)-algebra automorphisms
\[
T_i,\; T_i^*:\widehat{\mathcal A}\to \widehat{\mathcal A}
\]
defined by
\[
T_i(f_{j,m}) =
\begin{cases}
f_{j,m+\delta_{ij}}, & \text{if } d(i,j)\neq 1,\\[4pt]
\dfrac{q^{1/2} f_{j,m} f_{i,m}-q^{-1/2} f_{i,m} f_{j,m}}{q-q^{-1}},
& \text{if } d(i,j)=1,
\end{cases}
\]
and
\[
T_i^{*}(f_{j,m}) =
\begin{cases}
f_{j,m-\delta_{ij}}, & \text{if } d(i,j)\neq 1,\\[4pt]
\dfrac{q^{1/2} f_{i,m} f_{j,m}-q^{-1/2} f_{j,m} f_{i,m}}{q-q^{-1}},
& \text{if } d(i,j)=1.
\end{cases}
\]
Here \(d(i,j)\) denotes the number of edges between \(i\) and \(j\) in the Dynkin
diagram. The families \(\{T_i\}_{i\in I}\) and \(\{T_i^*\}_{i\in I}\) satisfy
the commutation and braid relations associated with the Cartan matrix \(C\).
Moreover,
\[
T_iT_i^*=T_i^*T_i=\mathrm{Id}.
\]
\end{proposition}

For an expression $\beta=(i_1,\ldots,i_r)$ of $b\in\Br^+$, set
$T_b:=T_{i_1}\circ\cdots\circ T_{i_r}$.
Proposition~\ref{pro:braidsym} makes this independent of the expression.

\begin{definition}
Let \(\beta=(i_1,\dots,i_r)\) be a word. For each \(k\in[r]\), define the
\emph{PBW root vector}
\[
P_k^\beta
:=
T_{i_1}\cdots T_{i_{k-1}}\bigl(q^{1/2}f_{i_k,0}\bigr).
\]
For \(\mathbf a=(a_1,\ldots,a_r)\in\mathbb N^r\), define
\[
P^\beta(\mathbf a)
:=
\left(\prod_{k=1}^{r}q^{a_k(a_k-1)/2}\right)
(P_r^\beta)^{a_r}\cdots(P_1^\beta)^{a_1}.
\]
Thus larger indices occur on the left, in accordance with the ordered
product convention in \cite{kashiwara2024braid}.
\end{definition}

\begin{definition}
For \(b\in\Br^+\), the \emph{bosonic extension algebra associated with \(b\)}
is defined by
\[
\widehat{\mathcal A}(b)
:=
T_b(\widehat{\mathcal A}_{<0})\cap \widehat{\mathcal A}_{\geq 0}.
\]
\end{definition}

\begin{proposition}[{\cite{oh2025pbw,kashiwara2024braid}}]
Let \(b\in\Br^+\), and let \(\beta\) be an expression of \(b\). Then
\[
\{P^\beta(\mathbf a)\mid \mathbf a\in\mathbb N^r\}
\]
forms a \(\mathbb Q(q^{1/2})\)-basis of \(\widehat{\mathcal A}(b)\). This basis
is called the \emph{PBW basis} associated with \(\beta\).
\end{proposition}

Following \cite[Definition~4.14]{kashiwara2024braid}, for
$\mathbf a,\mathbf b\in\mathbb N^r$ we write $\mathbf a\prec\mathbf b$ if
there exist indices \(k_0\leq k_1\) such that
\[
a_j=b_j \quad \text{for all } j<k_0,
\qquad
a_{k_0}<b_{k_0},
\]
and
\[
a_j=b_j \quad \text{for all } j>k_1,
\qquad
a_{k_1}<b_{k_1}.
\]

Let \(\mathbb K_0=\mathbb Z[q^{\pm1}]\). We denote by
\(\widehat{\mathcal A}(b)_{\mathbb K_0}\) the integral form used in
\cite{kashiwara2024braid}.

\begin{proposition}[{\cite{kashiwara2024braid}}]
For each \(\mathbf a\in\mathbb N^r\), there exists a unique element
\[
G^\beta(\mathbf a)\in \widehat{\mathcal A}(b)_{\mathbb K_0}
\]
satisfying:
\begin{itemize}
\item \(c(G^\beta(\mathbf a))=G^\beta(\mathbf a)\);
\item
\begin{equation}\label{eq:GPbfa}
G^\beta(\mathbf a)
=
P^\beta(\mathbf a)
+
\sum_{\mathbf b\prec\mathbf a}
f_{\mathbf b,\mathbf a}(q)\,P^\beta(\mathbf b),
\qquad
f_{\mathbf b,\mathbf a}(q)\in q\mathbb Z[q].
\end{equation}
\end{itemize}
The set
\[
\mathbb B(b):=\{G^\beta(\mathbf a)\mid \mathbf a\in\mathbb N^r\}
\]
is the \emph{global basis} of $\widehat{\mathcal A}(b)$.
The \emph{$\beta$-Lusztig parameter} of $G^\beta(\mathbf a)$ is
$\mathbf a$, denoted by $\mathbf a^\beta(G^\beta(\mathbf a))$.
\end{proposition}

For a homogeneous global basis element $G$ of weight $\alpha$, set
\[
    \widetilde G:=q^{-(\alpha,\alpha)/4}G.
\]
The elements $\widetilde G$ form the \emph{normalized global basis}; they
are invariant under the bar involution, whereas $G$ is invariant under $c$.
We write $\widetilde G^\beta(\mathbf a)$ for the normalization of
$G^\beta(\mathbf a)$. We extend the Lusztig-parameter notation by
\[
    \mathbf a^\beta(q^{s/2}G^\beta(\mathbf a)):=\mathbf a
    \qquad(s\in\mathbb Z).
\]
Thus normalization does not change parameters; the same convention applies
to products that are powers of $q^{1/2}$ times global basis elements.

\begin{remark}
By \cite[Lemma~4.17]{kashiwara2024braid}, the basis
\(\{G^\beta(\mathbf a)\}\) coincides with the global basis constructed in
\cite{kashiwara2025global}. The integral form
\(\widehat{\mathcal A}(b)_{\mathbb K_0}\) satisfies
\[
\widehat{\mathcal A}(b)_{\mathbb K_0}
\otimes_{\mathbb K_0}\mathbb Q(q^{1/2})
\cong
\widehat{\mathcal A}(b).
\]
For its construction, see \cite{kashiwara2025global}.
\end{remark}
For $\mathbf a\in\mathbb N^r$, put
$\operatorname{supp}(\mathbf a)=\{k:a_k\ne0\}$, with
$\min\varnothing=+\infty$ and $\max\varnothing=-\infty$.
Write $\mathbf a>\mathbf b$ if
$\min\operatorname{supp}(\mathbf a)>\max\operatorname{supp}(\mathbf b)$;
in particular, their supports are disjoint.

\begin{lemma}\label{lem:Gproductab}
If \(\bfa,\bb\in\mathbb N^r\) satisfy \(\bfa>\bb\), then
\[
G^\beta(\bfa)\cdot G^\beta(\bb)
=
G^\beta(\bfa+\bb)
+
\sum_{\mathbf{c}\prec\bfa+\bb}
g_{\mathbf{c},\bfa+\bb}(q)\,G^\beta(\mathbf{c}),
\quad
g_{\mathbf{c},\bfa+\bb}(q)\in q\mathbb Z[q].
\]
\end{lemma}

\begin{proof}
In the triangular expansion \eqref{eq:GPbfa}, lower parameters satisfy
\[
    \min\operatorname{supp}(\mathbf a')\geq
       \min\operatorname{supp}(\mathbf a),\qquad
    \max\operatorname{supp}(\mathbf b')\leq
       \max\operatorname{supp}(\mathbf b).
\]
Thus $\mathbf a'>\mathbf b'$ for every pair of terms from the expansions
of $G^\beta(\mathbf a)$ and $G^\beta(\mathbf b)$, including their leading
terms. The descending PBW product order gives
\[
    P^\beta(\mathbf a')P^\beta(\mathbf b')
    =P^\beta(\mathbf a'+\mathbf b').
\]
The order $\prec$ is preserved by addition: if either input is strictly
lower and the other is unchanged or lower, the first and last differing
coordinates of the sum are both smaller. Hence
$\mathbf a'+\mathbf b'\prec\mathbf a+\mathbf b$ unless both inputs
are leading. Multiplying the two triangular expansions therefore gives
\[
    G^\beta(\mathbf a)G^\beta(\mathbf b)
    =P^\beta(\mathbf a+\mathbf b)
      +\sum_{\mathbf c\prec\mathbf a+\mathbf b}
         h_{\mathbf c}(q)P^\beta(\mathbf c),
    \qquad h_{\mathbf c}(q)\in q\mathbb Z[q].
\]
The inverse PBW-to-global transition is also unitriangular with lower
coefficients in $q\mathbb Z[q]$. Substituting it proves the asserted
expansion with leading coefficient one.
\end{proof}
\begin{lemma}\label{lem:twoprod}
Let \(\bfa,\bb\in\mathbb N^r\). Suppose that
\[
G^\beta(\bfa)\,G^\beta(\bb)
=
\sum_{\mathbf c\in\mathbb N^r}
f_{\bfa,\bb}^{\mathbf c}(q)\,G^\beta(\mathbf c).
\]
Then the unique maximal index \(\mathbf c\), with respect to the order \(\prec\), for which
\(f_{\bfa,\bb}^{\mathbf c}(q)\neq 0\), is
\[
\mathbf c_{\max}=\bfa+\bb.
\]
\end{lemma}

\begin{proof}
By the Levendorskii--Soibelman formula
\cite[Lemma~5.5]{oh2025pbw}, for any
\(\mathbf d,\mathbf e\in\mathbb N^r\), we have
\[
P^\beta(\mathbf d)\,P^\beta(\mathbf e)
=
q^{A(\mathbf d,\mathbf e)}P^\beta(\mathbf d+\mathbf e)
+
\sum_{\mathbf c\prec \mathbf d+\mathbf e}
g_{\mathbf c}(q)\,P^\beta(\mathbf c),
\]
for some integer \(A(\mathbf d,\mathbf e)\).

Recall that the global basis is triangular with respect to the PBW basis:
\[
G^\beta(\bfa)
=
P^\beta(\bfa)
+
\sum_{\bfa'\prec\bfa}
f_{\bfa',\bfa}(q)\,P^\beta(\bfa'),
\]
and similarly
\[
G^\beta(\bb)
=
P^\beta(\bb)
+
\sum_{\bb'\prec\bb}
f_{\bb',\bb}(q)\,P^\beta(\bb').
\]
Multiplying these two expansions and applying the Levendorskii--Soibelman
formula to each product \(P^\beta(\bfa')P^\beta(\bb')\), we obtain
\[
G^\beta(\bfa)G^\beta(\bb)
=
q^{A(\bfa,\bb)}P^\beta(\bfa+\bb)
+
\sum_{\mathbf c\prec\bfa+\bb}
h_{\mathbf c}(q)\,P^\beta(\mathbf c).
\]
Indeed, if \(\bfa'\prec\bfa\) or \(\bb'\prec\bb\), then
\[
\bfa'+\bb'\prec\bfa+\bb
\]
by the definition of the order $\prec$. Hence no product involving a lower
PBW term can produce a PBW monomial indexed by \(\bfa+\bb\) or by an index
larger than \(\bfa+\bb\).

Now use the inverse triangular transition from the PBW basis to the global
basis. Since
\[
G^\beta(\mathbf d)
=
P^\beta(\mathbf d)
+
\sum_{\mathbf e\prec\mathbf d}
f_{\mathbf e,\mathbf d}(q)P^\beta(\mathbf e),
\]
we also have
\[
P^\beta(\mathbf d)
=
G^\beta(\mathbf d)
+
\sum_{\mathbf e\prec\mathbf d}
r_{\mathbf e,\mathbf d}(q)G^\beta(\mathbf e)
\]
for suitable coefficients \(r_{\mathbf e,\mathbf d}(q)\). Therefore
\[
G^\beta(\bfa)G^\beta(\bb)
=
q^{A(\bfa,\bb)}G^\beta(\bfa+\bb)
+
\sum_{\mathbf c\prec\bfa+\bb}
H_{\mathbf c}(q)G^\beta(\mathbf c).
\]
Since \(q^{A(\bfa,\bb)}\neq 0\), the coefficient of
\(G^\beta(\bfa+\bb)\) is nonzero, and every other index appearing in the
expansion is strictly smaller than \(\bfa+\bb\). Hence the unique maximal
index is \(\bfa+\bb\).
\end{proof}

\begin{lemma}\label{lem:exchange-parameters}
Fix a word $\omega$ of length $n$, and write $\widehat{\mathcal A}(\omega)$
for the algebra associated with the positive braid represented by $\omega$.
Suppose that $X,Y,N_+,N_-\in\widehat{\mathcal A}(\omega)$ are powers of
$q^{1/2}$ times global basis elements and that
\[
    XY=q^aN_++q^bN_-,\qquad a,b\in\tfrac12\mathbb Z.
\]
Then the two parameters $\mathbf a^\omega(N_+)$ and
$\mathbf a^\omega(N_-)$ are either equal or comparable for $\prec$, and
\[
    \mathbf a^\omega(X)+\mathbf a^\omega(Y)
    =\max_{\prec}\{\mathbf a^\omega(N_+),\mathbf a^\omega(N_-)\}.
\]
For $0\leq m\leq n$, let $\pi_m$ denote projection to the first $m$
coordinates and put $\nu_m=\pi_m\mathbf a^\omega$. Then
\begin{equation}\label{eq:prefix-exchange-parameters}
    \nu_m(Y)
    =\max_{\mathrm{lex}}\{\nu_m(N_+),\nu_m(N_-)\}-\nu_m(X),
\end{equation}
where $\max_{\mathrm{lex}}$ uses the left lexicographic order. For $m=0$
all projected vectors are the unique empty tuple.
\end{lemma}

\begin{proof}
By Lemma~\ref{lem:twoprod}, the left-hand side has a unique maximal
global-basis parameter, namely
$\mathbf a^\omega(X)+\mathbf a^\omega(Y)$.
The right-hand side consists of at most two global basis elements, and
its coefficients cannot cancel: if the two elements coincide, their
coefficient is a sum of two powers of $q^{1/2}$. Thus its two parameters
are comparable, and its maximum equals the maximum on the left.
If their first $m$ coordinates differ, the first differing coordinate
also determines their left lexicographic order and hence the larger
of the comparable full parameters. If those coordinates agree, either
full parameter has the same projection. Projecting the first identity
therefore gives \eqref{eq:prefix-exchange-parameters}.
\end{proof}

\begin{example}\label{exam:Delta}
Let $\Delta$ be the positive braid lift of $w_0$, and define $i^*$ by
$w_0(\alpha_i)=-\alpha_{i^*}$. The identity
$T_\Delta(f_{i,k})=f_{i^*,k+1}$ gives, for every $m\geq0$,
\[
    T_{\Delta^{m+1}}(\widehat{\mathcal A}_{<0})
    =\widehat{\mathcal A}[-\infty,m],\qquad
    \widehat{\mathcal A}(\Delta^{m+1})
    =\widehat{\mathcal A}[0,m].
\]
In particular,
$\widehat{\mathcal A}(\Delta)=\widehat{\mathcal A}[0]
\cong U_q^-(\mathfrak g)$, equivalently $U_q(\mathfrak n)$ under our
convention. For a reduced word of $w_0$, the PBW construction above is
the finite-type construction with the normalization fixed here.
\end{example}

\subsection{Transition maps of Lusztig parameters}

A \emph{$2$-move} exchanges adjacent letters $(i,j)$ with $c_{ij}=0$;
a \emph{$3$-move} replaces $(i,j,i)$ by $(j,i,j)$ with $c_{ij}=-1$.
Any two positive expressions of the same $b\in\Br^+$ are connected
by these moves.

\begin{definition}
For an expression $\beta$ of $b$, let
$\Psi_\beta:\mathbb N^r\to\mathbb B(b)$, $\mathbf a\mapsto G^\beta(\mathbf a)$.
For two expressions $\beta,\beta'$, define the \emph{transition map}
\[
\Psi_\beta^{\beta'}:=\Psi_{\beta'}^{-1}\circ\Psi_\beta:\NN^r\to\NN^r.
\]
\end{definition}
These maps are well-defined because the global basis is expression-independent.

\begin{theorem}[{\cite[Theorem~3.9]{bi2025cluster}}]\label{theo:braidmoves}
Let \(\beta\) and \(\beta'\) be two expressions of \(b\in\Br^+\).
Then the transition map \(\Psi_\beta^{\beta'}\) is given as follows.
\begin{enumerate}
    \item \emph{(2-move)}
    If \(\beta'\) is obtained from \(\beta\) by a \(2\)-move exchanging
    positions \(k\) and \(k+1\), then for
    \(\mathbf a=(a_1,\dots,a_r)\in\mathbb N^r\),
    \[
    \bigl(\Psi_\beta^{\beta'}(\mathbf a)\bigr)_s=
    \begin{cases}
        a_{k+1}, & s=k,\\
        a_k, & s=k+1,\\
        a_s, & \text{otherwise}.
    \end{cases}
    \]

    \item \emph{(3-move)}
    If \(\beta'\) is obtained from \(\beta\) by a \(3\)-move replacing
    \[
    (i_{k-1},i_k,i_{k+1})=(i,j,i)
    \]
    with
    \[
    (j,i,j),
    \]
    where \(c_{ij}=-1\), then
    \[
    \bigl(\Psi_\beta^{\beta'}(\mathbf a)\bigr)_s=
    \begin{cases}
        a_k+a_{k+1}-p, & s=k-1,\\
        p, & s=k,\\
        a_{k-1}+a_k-p, & s=k+1,\\
        a_s, & \text{otherwise},
    \end{cases}
    \qquad
    p=\min\{a_{k-1},a_{k+1}\}.
    \]
\end{enumerate}
\end{theorem}

\begin{proof}
\emph{The \(2\)-move case.}
Assume that \(\beta'\) is obtained from \(\beta\) by exchanging
\((i_k,i_{k+1})\), where \(c_{i_ki_{k+1}}=0\). Then
\[
P_k^\beta=P_{k+1}^{\beta'},
\qquad
P_{k+1}^\beta=P_k^{\beta'},
\]
and all other PBW root vectors agree. The two exchanged root vectors
commute because $c_{i_ki_{k+1}}=0$. Therefore
\[
P^\beta(\mathbf a)
=
P^{\beta'}\bigl(\Psi_\beta^{\beta'}(\mathbf a)\bigr).
\]
The defining triangular expansion and $c$-invariance of the global basis
therefore give, by uniqueness,
\[
G^\beta(\mathbf a)
=
G^{\beta'}\bigl(\Psi_\beta^{\beta'}(\mathbf a)\bigr).
\]

\medskip
\noindent\emph{The \(3\)-move case.}
Write $\mathbf j=(i,j,i)$ and $\mathbf j'=(j,i,j)$ for the subwords
at positions $k-1,k,k+1$.
In type \(A_2\), the Lusztig parameters of the global basis of
\(\widehat{\mathcal A}[0]\) agree with the usual Lusztig parameters for
\(U_q(\mathfrak n)\). Hence, by
\cite[Proposition~5.2]{kamnitzer2010mirkovic}, we have
\[
G^{\mathbf j}(\mathbf a)
=
G^{\mathbf j'}\!\left(\Psi_{\mathbf j}^{\mathbf j'}(\mathbf a)\right).
\]

Write
\[
\beta=(\beta_{<k-1},\mathbf j,\beta_{>k+1}),
\qquad
\beta'=(\beta_{<k-1},\mathbf j',\beta_{>k+1}).
\]
First assume that \(\mathbf a\) is supported on the positions
\(\{k-1,k,k+1\}\). Then
\[
P^\beta(\mathbf a)
=
T_{b_{<k-1}}\!\left(P^{\mathbf j}(\mathbf a)\right),
\]
and
\[
P^{\beta'}\!\left(\Psi_\beta^{\beta'}(\mathbf a)\right)
=
T_{b_{<k-1}}\!\left(
P^{\mathbf j'}\!\left(\Psi_{\mathbf j}^{\mathbf j'}(\mathbf a)\right)
\right).
\]
Using the triangular characterization of the global basis and the fact that
\(T_{b_{<k-1}}\) preserves the global basis
\cite[Theorem~3.7]{kashiwara2024braid}, we obtain
\[
G^\beta(\mathbf a)
=
T_{b_{<k-1}}\!\left(G^{\mathbf j}(\mathbf a)\right),
\]
and
\[
G^{\beta'}\!\left(\Psi_\beta^{\beta'}(\mathbf a)\right)
=
T_{b_{<k-1}}\!\left(
G^{\mathbf j'}\!\left(\Psi_{\mathbf j}^{\mathbf j'}(\mathbf a)\right)
\right).
\]
Therefore
\[
G^\beta(\mathbf a)
=
G^{\beta'}\!\left(\Psi_\beta^{\beta'}(\mathbf a)\right).
\]

If $\mathbf a$ is supported entirely in the prefix $[1,k-2]$ or
entirely in the suffix $[k+2,r]$, the PBW root vectors on that interval
are unchanged. Every lower term in the triangular expansion stays inside
the same interval, by the definition of $\prec$. The triangular
characterization therefore gives
\[
    G^\beta(\mathbf a)=G^{\beta'}(\mathbf a).
\]

Now let \(\mathbf a\) be arbitrary. Decompose
\[
\mathbf a
=
\mathbf a_{<k-1}
+
\mathbf a_{[k-1,k+1]}
+
\mathbf a_{>k+1}.
\]
By the two special cases above, we have
\[
G^\beta(\mathbf a_{<k-1})
=
G^{\beta'}(\mathbf a_{<k-1}),
\]
\[
G^\beta(\mathbf a_{>k+1})
=
G^{\beta'}(\mathbf a_{>k+1}),
\]
and
\[
G^\beta(\mathbf a_{[k-1,k+1]})
=
G^{\beta'}\!\left(
\Psi_\beta^{\beta'}(\mathbf a_{[k-1,k+1]})
\right).
\]
Therefore the following two products are equal:
\[
G^\beta(\mathbf a_{>k+1})
G^\beta(\mathbf a_{[k-1,k+1]})
G^\beta(\mathbf a_{<k-1})
=
G^{\beta'}(\mathbf a_{>k+1})
G^{\beta'}\!\left(
\Psi_\beta^{\beta'}(\mathbf a_{[k-1,k+1]})
\right)
G^{\beta'}(\mathbf a_{<k-1}).
\]
Let $H$ denote this common product, and let $\mathbf a^{\mathrm{new}}$
be the vector obtained from $\mathbf a$ by the displayed rank-two
transformation. Applying Lemma~\ref{lem:Gproductab} successively to the
three consecutively supported factors gives
\[
    H=G^\beta(\mathbf a)
       +\sum_{\mathbf d\prec\mathbf a}
         h_{\mathbf d}(q)G^\beta(\mathbf d),
    \qquad h_{\mathbf d}(q)\in q\mathbb Z[q],
\]
and
\[
    H=G^{\beta'}(\mathbf a^{\mathrm{new}})
       +\sum_{\mathbf e\prec\mathbf a^{\mathrm{new}}}
         h'_{\mathbf e}(q)G^{\beta'}(\mathbf e),
    \qquad h'_{\mathbf e}(q)\in q\mathbb Z[q].
\]
Both are expansions in the same, expression-independent global basis.
Reducing their coefficients modulo $q$, exactly one basis element has
nonzero coefficient on either side. Consequently,
\[
    G^\beta(\mathbf a)=G^{\beta'}(\mathbf a^{\mathrm{new}}).
\]
By the definition of the transition map,
$\Psi_\beta^{\beta'}(\mathbf a)=\mathbf a^{\mathrm{new}}$.
This completes the proof.
\end{proof}
\begin{example}\label{exm:010}
For a $3$-move, writing $\mathbf b=\Psi_\beta^{\beta'}(\mathbf a)$,
the following local transitions will be used below:
\[
\begin{array}{c|cccc}
 (a_{k-1},a_k,a_{k+1})&(1,0,1)&(1,0,0)&(0,1,0)&(0,0,1)\\
 \hline
 (b_{k-1},b_k,b_{k+1})&(0,1,0)&(0,0,1)&(1,0,1)&(1,0,0).
\end{array}
\]
\end{example}

\subsection{Cluster structure on bosonic extension algebras}
\begin{definition}
Let $\beta=(i_1,\ldots,i_r)$, and let $E_1,\ldots,E_r$ be the standard
basis of $\mathbb Z^r$. For $1\leq a\leq b\leq r$, set
\[
    \beta\{a,b]:=\sum_{\substack{a\leq k\leq b\\i_k=i_b}}E_k,
    \qquad
    \beta[a,b\}:=\sum_{\substack{a\leq k\leq b\\i_k=i_a}}E_k.
\]
If $i_a=i_b$, denote their common value by $\beta[a,b]$.
The normalized quantum minors are
\[
    D_k^\beta:=\widetilde G^\beta\bigl(\beta\{1,k]\bigr)
    \qquad(k\in[r]).
\]
\end{definition}

\begin{theorem}[{\cite[Theorem~9.7]{kashiwara2025monoidal}; \cite{bi2025cluster}}]
Let $b\in\Br^+$, and let $\beta=(i_1,\ldots,i_r)$ be an expression of $b$.
Use the quantum parameter $t$ and the conventions fixed above, with the
coefficient embedding
\[
    \mathbb Z[t^{\pm1/2}]\longrightarrow\mathbb Q(q^{1/2}),
    \qquad t^{1/2}\longmapsto q^{-1/2}.
\]
Then
\[
    \mathbf s(\beta)
    =\bigl(\{D_k^\beta\}_{k\in[r]},L_\beta,B_\beta,K_{\rm ex}\bigr)
\]
is a quantum seed in $\widehat{\mathcal A}(b)$. More precisely, the
assignment $X_k\mapsto D_k^\beta$ induces an algebra isomorphism
\[
    \overline{\mathcal A}_t\bigl(\mathbf s(\beta)\bigr)
    \otimes_{\mathbb Z[t^{\pm1/2}]}\mathbb Q(q^{1/2})
    \xrightarrow{\ \sim\ }\widehat{\mathcal A}(b).
\]
Every normalized quantum cluster monomial with nonnegative frozen
exponents belongs to the normalized global basis of
$\widehat{\mathcal A}(b)$.
\end{theorem}
Braid moves relate these normalized minors as follows.

\begin{proposition}\label{prop:Dbraid}
Let \(\beta,\beta'\) be two expressions of an element \(b\in\Br^+\).
\begin{enumerate}
    \item If \(\beta'\) is obtained from \(\beta\) by a \(2\)-move exchanging
    \((i_k,i_{k+1})\), then
    \[
    D^{\beta'}_k = D^\beta_{k+1},\qquad
    D^{\beta'}_{k+1} = D^\beta_k,\qquad
    D^{\beta'}_j = D^\beta_j \quad \text{for all } j\neq k,k+1.
    \]

    \item If \(\beta'\) is obtained from \(\beta\) by a \(3\)-move replacing
    \[
    (i_{k-1},i_k,i_{k+1})=(i,j,i)
    \]
    with
    \[
    (j,i,j),
    \]
    where \(c_{ij}=-1\), then
    \[
    D^{\beta'}_j=
    \begin{cases}
        D^\beta_j, & j\neq k-1,k,k+1,\\
        D^\beta_{k+1}, & j=k,\\
        D^\beta_k, & j=k+1,\\
        \mu_{k-1}(D^\beta_{k-1}), & j=k-1.
    \end{cases}
    \]
\end{enumerate}
\end{proposition}

\begin{proof}
The \(2\)-move case follows directly from the formula for
\(\Psi_{\beta'}^\beta\) in Theorem~\ref{theo:braidmoves}. Indeed,
\[
D^{\beta'}_j
=
\widetilde G^{\beta'}(\beta'\{1,j])
=
\widetilde G^\beta\!\left(\Psi_{\beta'}^\beta(\beta'\{1,j])\right),
\]
and the \(2\)-move formula exchanges the two relevant coordinates.

We now prove the \(3\)-move case. Assume that \(\beta'\) is obtained from
\(\beta\) by replacing the subword
\[
(i_{k-1},i_k,i_{k+1})=(i,j,i)
\]
with
\[
(j,i,j).
\]

If \(s\neq k-1,k,k+1\), then the transition formulas in
Theorem~\ref{theo:braidmoves}, equivalently Example~\ref{exm:010}, give
\[
\Psi_{\beta'}^\beta(\beta'\{1,s])=\beta\{1,s].
\]
Hence
\[
D^{\beta'}_s
=
\widetilde G^{\beta'}(\beta'\{1,s])
=
\widetilde G^\beta(\beta\{1,s])
=
D^\beta_s.
\]

For \(s=k+1\), the local part of \(\beta'\{1,k+1]\) at
\((k-1,k,k+1)\) is \((1,0,1)\). By the \(3\)-move transition formula,
\[
\Psi_{\beta'}^\beta(\beta'\{1,k+1])=\beta\{1,k].
\]
Therefore
\[
D^{\beta'}_{k+1}=D^\beta_k.
\]
Similarly, the local part of \(\beta'\{1,k]\) is \((0,1,0)\), and the
transition formula gives
\[
\Psi_{\beta'}^\beta(\beta'\{1,k])=\beta\{1,k+1].
\]
Thus
\[
D^{\beta'}_k=D^\beta_{k+1}.
\]

It remains to identify \(D^{\beta'}_{k-1}\). The local part of
\(\beta'\{1,k-1]\) is
\[
(1,0,0)
\]
at the positions \((k-1,k,k+1)\) of \(\beta'\). Applying
\(\Psi_{\beta'}^\beta\), we obtain the local vector
\[
(0,0,1)
\]
at the corresponding positions \((k-1,k,k+1)\) of \(\beta\). Hence
\[
\Psi_{\beta'}^\beta(\beta'\{1,k-1])
=
\bfa^\beta(D^\beta_{k^-})+\epsilon_{k+1},
\]
where \(\epsilon_{k+1}\) denotes the unit vector at position \(k+1\), and we
use the convention \(D^\beta_{-\infty}=1\) if \(k^-\) does not exist.

We claim that this vector is the \(\beta\)-Lusztig parameter of
\(\mu_{k-1}(D^\beta_{k-1})\).

By the exchange relation at the vertex \(k-1\), the local quiver gives
\[
D^\beta_{k-1}\,\mu_{k-1}(D^\beta_{k-1})
=
q^{A}D^\beta_kD^\beta_{(k-1)^-}
+
q^{B}D^\beta_{k+1}D^\beta_{k^-}
\]
for some \(A,B\in\tfrac12\mathbb Z\). Here, if \((k-1)^-\) or \(k^-\) does not exist, the
corresponding factor is omitted.

Let
\[
\mathbf u
=
\bfa^\beta(D^\beta_kD^\beta_{(k-1)^-}),
\qquad
\mathbf v
=
\bfa^\beta(D^\beta_{k+1}D^\beta_{k^-}).
\]
The two vectors agree outside $\{k-1,k,k+1\}$, and their local
coordinates are
\[
    \mathbf u|_{\{k-1,k,k+1\}}=(0,1,0),\qquad
    \mathbf v|_{\{k-1,k,k+1\}}=(1,0,1).
\]
Thus both the first and the last differing coordinates are smaller for
$\mathbf u$, so $\mathbf u\prec\mathbf v$.
Thus the maximal Lusztig parameter on the right-hand side of the exchange
relation is
\[
\mathbf v
=
\bfa^\beta(D^\beta_{k+1})+\bfa^\beta(D^\beta_{k^-}).
\]

On the other hand, the maximal Lusztig parameter of the left-hand side is
\[
\bfa^\beta(D^\beta_{k-1})
+
\bfa^\beta\bigl(\mu_{k-1}(D^\beta_{k-1})\bigr),
\]
again by Lemma~\ref{lem:twoprod}. Comparing maximal parameters in the exchange
relation gives
\[
\bfa^\beta(D^\beta_{k-1})
+
\bfa^\beta\bigl(\mu_{k-1}(D^\beta_{k-1})\bigr)
=
\bfa^\beta(D^\beta_{k+1})
+
\bfa^\beta(D^\beta_{k^-}).
\]
Since
\[
\bfa^\beta(D^\beta_{k+1})
=
\bfa^\beta(D^\beta_{k-1})+\epsilon_{k+1},
\]
we obtain
\[
\bfa^\beta\bigl(\mu_{k-1}(D^\beta_{k-1})\bigr)
=
\epsilon_{k+1}+\bfa^\beta(D^\beta_{k^-}).
\]
This is precisely
\[
\Psi_{\beta'}^\beta(\beta'\{1,k-1]).
\]
Therefore
\[
\mu_{k-1}(D^\beta_{k-1})
=
\widetilde G^\beta\!\left(\Psi_{\beta'}^\beta(\beta'\{1,k-1])\right)
=
\widetilde G^{\beta'}(\beta'\{1,k-1])
=
D^{\beta'}_{k-1}.
\]
This proves the \(3\)-move case and hence the proposition.
\end{proof}

\section{ New subalgebras of bosonic extension algebras}
\label{sec:subalgebra}

In this section, we introduce a distinguished subalgebra of the bosonic
extension algebra.

\subsection{The intersection subalgebra}
\label{sec:algebracA}

Let \(b\in\Br^+\). By the Garside property \cite[Corollary~7.3]{oh2025pbw}, there exist an element
\(u\in\Br^+\) and an integer \(m>0\) such that
\begin{equation}\label{eq:Delta-factor}
    bu=\Delta^m,
\end{equation}
where \(\Delta\) denotes the positive braid lift of the longest Weyl group
element \(w_0\).

Let \(v\le \delta(b)\), and let
\[
\beta=(i_1,\dots,i_r)
\]
be an expression of \(b\). Let
\[
\beta_v=(i_{p_1},\dots,i_{p_\ell})
\]
be the leftmost reduced subexpression of \(\beta\) representing \(v\). Since
\(v\le w_0\), we may extend \(\beta_v\) to a reduced expression of \(w_0\);
we denote such an extension by \(\overline w_0\).

This reduced expression determines an infinite word, as in
\eqref{eq:dotw0}. We denote the resulting infinite word by
\begin{equation}
    \dot{\beta}_v
    =
    (i_{p_1},\dots,i_{p_\ell},\dots,
    i^*_{p_1},\dots,i^*_{p_\ell},\dots),
\end{equation}
where \(i^*\in I\) is determined by
\[
w_0(\alpha_i)=-\alpha_{i^*}.
\]

For a Weyl group element \(v\in W\), we write \(T_v\) for the braid symmetry
associated with the positive braid lift of \(v\). Equivalently, if
\[
v=s_{j_1}\cdots s_{j_\ell}
\]
is a reduced expression, then
\[
T_v:=T_{j_1}\cdots T_{j_\ell}.
\]
This is independent of the chosen reduced expression.

\begin{definition}
Let \(b\in\Br^+\), and let \(v\le \delta(b)\). We define
\[
\widehat{\mathcal A}_{v,b}
:=
\widehat{\mathcal A}(b)\cap T_v\bigl(\widehat{\mathcal A}_{\ge0}\bigr).
\]
\end{definition}

The following proposition gives an intrinsic characterization of
\(\widehat{\cA}_{v,b}\) in terms of these Lusztig parameters.

\begin{proposition}\label{prop:Avb-characterization}
Let \(x\in \widehat{\cA}(b)\) be a global basis element. Then
\[
x\in \widehat{\cA}_{v,b}
\quad\Longleftrightarrow\quad
\bfa^{\dot{\beta}_v}_k(x)=0
\quad \text{for all } k\in[\ell(v)].
\]
\end{proposition}

\begin{proof}
Write the chosen reduced word of \(w_0\) as \(\beta_v\tau\), and
put \(m=\ell(v)\). The word \(\tau\beta_v^*\) is also reduced for
\(w_0\): it has length \(\ell(w_0)\) and represents
\[
 (v^{-1}w_0)(w_0vw_0)=w_0.
\]
The tail of \(\dot\beta_v\) after its first \(m\) letters is the
alternating-star word built from \(\tau\beta_v^*\). By
Example~\ref{exam:Delta}, its finite Garside prefixes provide PBW
bases of \(\widehat{\cA}[0,M]\) for all \(M\). Applying \(T_v\)
to these root vectors gives \(P_k^{\dot\beta_v}\) for \(k>m\).
Consequently, the PBW monomials with zero first \(m\) coordinates
form a basis of \(T_v(\widehat{\cA}_{\ge0})\).

The set of such parameters is downward closed in the triangular
order. Indeed, if \(\mathbf a\) has zero first \(m\) coordinates
and \(\mathbf b\prec\mathbf a\), let \(k_0\) be their first
differing coordinate. Then \(b_{k_0}<a_{k_0}\), so nonnegativity
forces \(k_0>m\). Hence \(b_i=0\) for \(i\le m\).

Both the triangular expansion of \(G(\mathbf a)\) in the PBW basis
and its inverse therefore preserve this parameter set. It follows
that \(T_v(\widehat{\cA}_{\ge0})\) is also spanned by precisely the
global basis elements with zero first \(m\) coordinates. All these
expansions can be made in a finite Garside prefix. Since \(x\) is
itself a global basis element, its membership in this span is
equivalent to the stated vanishing. Intersecting with
\(\widehat{\cA}(b)\) proves the assertion.
\end{proof}

\begin{corollary}\label{cor:Avb-based}
The algebra $\widehat{\cA}_{v,b}$ has as a basis the global basis elements
$x\in\mathbb B(b)$ satisfying
\[
    a_k^{\dot\beta_v}(x)=0\qquad(1\leq k\leq\ell(v)).
\]
The same statement holds with the normalized global basis.
\end{corollary}
\begin{proof}
The proof of Proposition~\ref{prop:Avb-characterization} shows that
$T_v(\widehat{\cA}_{\ge0})$ is spanned by the indicated subset of the
ambient global basis. The algebra $\widehat{\cA}(b)$ is likewise
spanned by its global basis, viewed as a subset of that ambient basis.
The intersection of two subspaces spanned by subsets of one basis is
spanned by the intersection of those subsets. Normalization rescales
basis elements by nonzero scalars and does not alter either subspace.
\end{proof}

\subsection{Lusztig parameters for two expressions}
We now study the \(\dot{\beta}_v\)-Lusztig parameters of the cluster variables
\(D_k^\beta\). Choose \(M\ge0\) and an expression \(\gamma\) such that
\(
\beta\gamma
\)
is an expression of \(\Delta^{M+1}\). We denote this extended word by
\(
\overline{\beta}:=\beta\gamma .
\)
By Example~\ref{exam:Delta}, we have
\[
\widehat{\mathcal A}(\Delta^{M+1})
=
\widehat{\mathcal A}[0,M],
\]
and hence
\[
\widehat{\mathcal A}(b)\subseteq \widehat{\mathcal A}[0,M].
\]

On the other hand, the infinite word \(\dot{\beta}_v\) was constructed from a
reduced expression of \(w_0\). Let \(\dot{\beta}^M\) denote the finite prefix of
\(\dot{\beta}_v\) corresponding to the first \(M+1\) copies of \(\Delta\). Thus
\(\dot{\beta}^M\) is an expression of \(\Delta^{M+1}\).

Since both \(\overline{\beta}\) and \(\dot{\beta}^M\) are expressions of the
same positive braid element \(\Delta^{M+1}\), they are related by a finite
sequence of \(2\)-moves and \(3\)-moves.
\begin{lemma}\label{lem:beta3move}
Let \(v\le \delta(b)\), and let
\(
\beta=(i_1,\dots,i_r)
\)
be an expression of \(b\in \Br^+\). Let
\(
\beta_v=(p_1,\dots,p_m)
\)
be the sequence of positions of the leftmost reduced subexpression of \(v\) in
\(\beta\), where \(m=\ell(v)\). Then the following hold.

\begin{enumerate}
\item For \(i\in I\), one has
\[
s_iv<v
\quad\Longleftrightarrow\quad
v^{-1}(\alpha_i)\in R^-.
\]

\item Suppose that \(\beta'\) is obtained from \(\beta\) by a \(3\)-move
\[
(i_{k-1},i_k,i_{k+1})=(i,j,i)
\quad\longleftrightarrow\quad
(j,i,j),
\]
where \(c_{ij}=-1\). Then, in the following cases, the sequence of positions
\(\beta'_v=(p'_1,\dots,p'_m)\) is given as follows.

\begin{itemize}
\item[(1)] If
\[
\{p_1,\dots,p_m\}\cap\{k-1,k,k+1\}=\{k-1=p_t\},
\]
then
\[
p'_s=p_s \quad \text{for }s\neq t,
\qquad
p'_t=k.
\]

\item[(2)] If
\[
\{p_1,\dots,p_m\}\cap\{k-1,k,k+1\}
=
\{k-1=p_t,\ k=p_{t+1}\},
\]
then
\[
p'_s=p_s \quad \text{for }s\neq t,t+1,
\qquad
p'_t=k,\quad p'_{t+1}=k+1.
\]

\item[(3)] If
\[
\{p_1,\dots,p_m\}\cap\{k-1,k,k+1\}=\{k=p_t\},
\]
then
\[
p'_s=p_s \quad \text{for }s\neq t,
\qquad
p'_t=k-1.
\]

\item[(4)] If
\[
\{p_1,\dots,p_m\}\cap\{k-1,k,k+1\}
=
\{k=p_t,\ k+1=p_{t+1}\},
\]
then
\[
p'_s=p_s \quad \text{for }s\neq t,t+1,
\qquad
p'_t=k-1,\quad p'_{t+1}=k.
\]
\end{itemize}
\end{enumerate}
\end{lemma}

\begin{proof}
The first assertion is the standard left-descent criterion in Coxeter theory.
Indeed,
\[
s_i v<v
\]
if and only if \(s_i\) is a left descent of \(v\), which is equivalent to
\[
v^{-1}(\alpha_i)\in R^-.
\]

We prove the first two cases in (2). The last two follow by applying the same
argument to the inverse \(3\)-move.

\medskip
\noindent
\emph{First case.}
Assume
\[
\{p_1,\dots,p_m\}\cap\{k-1,k,k+1\}
=
\{k-1=p_t\}.
\]
Thus the local letter \(i_{k-1}=i\) is selected, while \(i_k=j\) and
\(i_{k+1}=i\) are not selected.

Set
\[
z:=v_{k-1}^{-1}v.
\]
Then
\[
v_{k-2}^{-1}v=s_i z.
\]
Since the letter at position \(k-1\) is selected, Lemma~\ref{lem:vk} gives
\[
z<s_i z.
\]
Equivalently,
\[
s_i z>z.
\]
Since the letter at position \(k\) is not selected, Lemma~\ref{lem:vk} gives
\[
s_j z>z.
\]
By part (1), we obtain
\[
z^{-1}(\alpha_i)\in R^+,
\qquad
z^{-1}(\alpha_j)\in R^+.
\]

In \(\beta'\), the local word is \((j,i,j)\). The candidate subexpression
obtained by replacing the selected position \(k-1\) in \(\beta\) by the
position \(k\) in \(\beta'\) still represents \(v\). Thus, by leftmostness,
\(\beta'_v\) is lexicographically no larger than this candidate.

We claim that \(\beta'_v\) cannot select position \(k-1\). Suppose otherwise.
Then, before position \(k-1\) in \(\beta'\), the remaining factor is \(s_i z\),
and selecting the local letter \(j\) gives, by Lemma~\ref{lem:vk},
\[
s_js_i z<s_i z.
\]
By part (1), this is equivalent to
\[
z^{-1}s_i(\alpha_j)\in R^-.
\]
Since \(c_{ij}=-1\), we have
\[
s_i(\alpha_j)=\alpha_i+\alpha_j.
\]
Therefore
\[
z^{-1}s_i(\alpha_j)
=
z^{-1}(\alpha_i)+z^{-1}(\alpha_j)\in R^-.
\]
This is impossible, because both \(z^{-1}(\alpha_i)\) and
\(z^{-1}(\alpha_j)\) are positive roots, and their sum is the root
\(z^{-1}s_i(\alpha_j)\). Hence \(p'_t\neq k-1\).

The candidate has \(p'_t=k\). Since \(\beta'_v\) is lexicographically no larger
than the candidate and cannot choose \(k-1\), it follows that
\[
p'_t=k.
\]

We next show that \(k+1\) is not selected. If \(p'_{t+1}=k+1\), then after the
position \(k\) has been selected, the remaining factor is \(z\), and selecting
the letter \(j\) at position \(k+1\) would imply
\[
s_jz<z
\]
by Lemma~\ref{lem:vk}. This contradicts \(s_jz>z\). Hence no further position
in \(\{k-1,k,k+1\}\) is selected.

Since \(\beta\) and \(\beta'\) coincide outside the local interval, the
remaining selected positions agree:
\[
p'_s=p_s\qquad(s\neq t).
\]
This proves the first case.

\medskip
\noindent
\emph{Second case.}
Assume
\[
\{p_1,\dots,p_m\}\cap\{k-1,k,k+1\}
=
\{k-1=p_t,\ k=p_{t+1}\}.
\]
Thus the local selected subword in \(\beta\) is \((i,j)\), and the position
\(k+1\) is not selected.

Set
\[
z:=v_k^{-1}v.
\]
Then
\[
v_{k-1}^{-1}v=s_jz,
\qquad
v_{k-2}^{-1}v=s_is_jz.
\]
Since the position \(k+1\) is not selected, Lemma~\ref{lem:vk} gives
\[
s_iz>z.
\]
By part (1), we have
\[
z^{-1}(\alpha_i)\in R^+.
\]

The candidate subexpression in \(\beta'\) selects the local positions
\[
k,\quad k+1,
\]
which again gives the local product \(s_is_j\). Hence the candidate has
\[
p'_t=k,\qquad p'_{t+1}=k+1.
\]

We claim that \(\beta'_v\) cannot select \(k-1\). Suppose otherwise. Then the
local letter at \(k-1\) in \(\beta'\) is \(j\), and before this position the
remaining factor is \(s_is_jz\). By Lemma~\ref{lem:vk}, selecting this letter
would imply
\[
s_js_is_jz<s_is_jz.
\]
By part (1), this gives
\[
z^{-1}s_js_i(\alpha_j)\in R^-.
\]
Since \(c_{ij}=-1\), we have
\[
s_js_i(\alpha_j)=\alpha_i.
\]
Thus
\[
z^{-1}(\alpha_i)\in R^-,
\]
contradicting \(z^{-1}(\alpha_i)\in R^+\). Hence \(p'_t\neq k-1\).

Since the candidate begins at \(k\), leftmostness forces
\[
p'_t=k.
\]
Moreover, the next selected position must be \(k+1\); otherwise the candidate
would be lexicographically smaller. Hence
\[
p'_{t+1}=k+1.
\]

Outside the local interval, \(\beta\) and \(\beta'\) coincide. Therefore
Lemma~\ref{lem:vk} gives
\[
p'_s=p_s
\qquad
\text{for all }s\neq t,t+1.
\]
This proves the second case.

The third and fourth cases are obtained by applying the first and second cases
to the inverse \(3\)-move
\[
(j,i,j)\longleftrightarrow(i,j,i).
\]
This completes the proof.
\end{proof}

\begin{remark}\label{rem:dotbetamoves}
Fix a reduced word $\tau$ for $v^{-1}w_0$. For every expression $\eta$
under consideration, let $\eta_v$ denote the selected word of its
leftmost reduced subexpression for $v$, and construct $\dot\eta_v$
from the reduced expression $\eta_v\tau$ of $w_0$.

Under a $2$-move, the selected word is unchanged unless both local
positions are selected; in that case it undergoes the same $2$-move.
Under a $3$-move, Lemma~\ref{lem:beta3move} shows that the selected word
is unchanged unless all three local positions are selected; in that
case it undergoes the same $3$-move.

Thus, when the selected word is unchanged, the chosen infinite words
are identical. When it changes by a braid move, the corresponding
move, with the appropriate star on the colors, occurs in every
$\Delta$-block of the infinite words. Their prefixes representing
$\Delta^{M+1}$ are therefore related by $M+1$ such moves. Only the move
in the first block can affect their first $\ell(v)$ Lusztig
coordinates; all the other moves fix those coordinates by
Theorem~\ref{theo:braidmoves}.
\end{remark}

\begin{proposition}\label{pro:bfadotbeta}
Let $v\leq\delta(b)$, and let $\beta=(i_1,\ldots,i_r)$ be an expression
of $b\in\Br^+$. Write $(p_1,\ldots,p_m)$ for the selected positions of
the leftmost reduced subexpression of $v$ in $\beta$, where
$m=\ell(v)$. Then
\[
    a_j^{\dot\beta_v}(D_k^\beta)
    =a_{p_j}^{\beta}(D_k^\beta)
    \qquad(1\leq j\leq m,\ 1\leq k\leq r).
\]
\end{proposition}

\begin{proof}
Choose an extension $\overline\beta=\beta\gamma$ representing
$\Delta^{M+1}$, and put $n=(M+1)\ell(w_0)$. Fix the reduced word $\tau$
for $v^{-1}w_0$ used in the chosen completion $\beta_v\tau$ defining
$\dot\beta_v$, as in Remark~\ref{rem:dotbetamoves}. For an expression
$\eta$ of $\Delta^{M+1}$, let $(p_1^\eta,\ldots,p_m^\eta)$ denote its
selected positions, and define
\[
    \nu_\eta(x):=\pi_m\mathbf a^{\dot\eta_v^M}(x),\qquad
    r_\eta(\mathbf a):=(a_{p_1^\eta},\ldots,a_{p_m^\eta}),
\]
where $\dot\eta_v^M$ is the finite prefix constructed from $\eta_v\tau$.
We prove
\begin{equation}\label{eq:prefix-minor-claim}
    \nu_\eta(D_s^\eta)
    =r_\eta\bigl(\mathbf a^\eta(D_s^\eta)\bigr)
    \qquad(1\leq s\leq n)
\end{equation}
by transport along braid moves.

For $\eta=\dot\beta_v^M$, the first $m$ letters are already a reduced
word for $v$, so the selected positions are $p_j^\eta=j$. With the
fixed completion $\tau$, the associated word is $\eta$ itself, and
\eqref{eq:prefix-minor-claim} is immediate. Connect this word to
$\overline\beta$ by $2$- and $3$-moves. It remains to check that the
claim is preserved by each move.

For a $2$-move, Proposition~\ref{prop:Dbraid} interchanges the two
local minors, and Theorem~\ref{theo:braidmoves} interchanges the two
local parameter coordinates. If both positions are selected, the
same interchange occurs among the first $m$ coordinates of the
associated word; otherwise that associated prefix is unchanged.
Remark~\ref{rem:dotbetamoves} shows that the moves in subsequent
$\Delta$-blocks do not affect these coordinates. This proves the
required transport for a $2$-move.

Now suppose that $\eta'$ is obtained from $\eta$ by
\[
    (i_{k-1},i_k,i_{k+1})=(i,j,i)\longleftrightarrow(j,i,j).
\]
For this local calculation write
\[
    S=\{p_1^\eta,\ldots,p_m^\eta\}\cap\{k-1,k,k+1\},\qquad
    S'=\{p_1^{\eta'},\ldots,p_m^{\eta'}\}\cap\{k-1,k,k+1\}.
\]
Use $1,2,3$ as abbreviations for the local positions $k-1,k,k+1$.
The possible pairs $(S,S')$ are
\begin{equation}\label{eq:selected-local-patterns}
\begin{array}{c|cccccc}
 S & \varnothing & \{1\} & \{2\} & \{1,2\} & \{2,3\} & \{1,2,3\}\\
 \hline
 S'& \varnothing & \{2\} & \{1\} & \{2,3\} & \{1,2\} & \{1,2,3\}.
\end{array}
\end{equation}
These follow from Lemma~\ref{lem:beta3move}. The set $\{3\}$ cannot
occur by leftmostness, and $\{1,3\}$ cannot occur because the selected
word is reduced. Selected positions outside this interval agree.

First consider a nonexceptional minor $D_s^{\eta'}$ with $s\ne k-1$.
By Proposition~\ref{prop:Dbraid}, it is one of the old minors, so the
induction hypothesis computes $\nu_\eta(D_s^{\eta'})$. Its local
parameters in the two words have one of the following forms:
\begin{equation}\label{eq:regular-local-parameters}
\begin{array}{c|ccc}
 \mathbf a^\eta(D_s^{\eta'})|_{\{k-1,k,k+1\}}
     &(0,0,0)&(1,0,1)&(0,1,0)\\
 \hline
 \mathbf a^{\eta'}(D_s^{\eta'})|_{\{k-1,k,k+1\}}
     &(0,0,0)&(0,1,0)&(1,0,1).
\end{array}
\end{equation}
All other parameter coordinates are unchanged. For each of the first
five columns of \eqref{eq:selected-local-patterns}, restricting the
first row of \eqref{eq:regular-local-parameters} to $S$ gives the same
tuple as restricting the second row to $S'$. The associated selected
word, and hence its prefix parameters, is unchanged in these cases.
If all three positions are selected, both the associated prefix and
the local parameter undergo the same rank-two transition. Thus
\eqref{eq:prefix-minor-claim} holds for every nonexceptional minor.

It remains to consider
\[
    Y:=D_{k-1}^{\eta'}=\mu_{k-1}(D_{k-1}^{\eta}).
\]
The exchange relation is
\begin{equation}\label{eq:exchangeD}
    D_{k-1}^{\eta}Y
    =q^A D_k^\eta D_{(k-1)^-}^\eta
      +q^B D_{k+1}^\eta D_{k^-}^\eta,
    \qquad A,B\in\tfrac12\mathbb Z,
\end{equation}
where a missing predecessor contributes the factor $1$. Set
\[
    N_+=D_k^\eta D_{(k-1)^-}^\eta,\qquad
    N_-=D_{k+1}^\eta D_{k^-}^\eta.
\]
These monomials and both exchange variables are powers of $q^{1/2}$
times global basis elements. Lemma~\ref{lem:exchange-parameters}
therefore gives the identity of \emph{prefix} parameters
\begin{equation}\label{eq:k-1inpi}
    \nu_\eta(Y)
    =\max_{\mathrm{lex}}\{\nu_\eta(N_+),\nu_\eta(N_-)\}
       -\nu_\eta(D_{k-1}^{\eta}).
\end{equation}
By Lemma~\ref{lem:twoprod}, parameters add in each old cluster
monomial. Hence the induction hypothesis also computes
$\nu_\eta(N_+)$ and $\nu_\eta(N_-)$ by restricting their
$\eta$-parameters to the selected positions.

Write
\[
    \gamma_i:=\mathbf a^\eta(D_{(k-1)^-}^\eta),\qquad
    \gamma_j:=\mathbf a^\eta(D_{k^-}^\eta),\qquad
    \gamma:=\gamma_i+\gamma_j.
\]
These vectors are supported strictly before $k-1$. The three relevant
parameters are
\begin{align*}
    \mathbf a^\eta(N_+)&=\gamma+E_k,\\
    \mathbf a^\eta(N_-)&=\gamma+E_{k-1}+E_{k+1},\\
    \mathbf a^\eta(D_{k-1}^{\eta})&=\gamma_i+E_{k-1}.
\end{align*}
Thus, outside the selected local positions, \eqref{eq:k-1inpi}
gives the restriction of $\gamma_j$. On the selected local positions,
its values are as follows:
\begin{equation}\label{eq:exceptional-prefix-table}
\begin{array}{c|cccccc}
 S & \varnothing & \{1\} & \{2\} & \{1,2\} & \{2,3\} & \{1,2,3\}\\
 \hline
 \nu_\eta(Y)\text{ on }S
   & () & (0) & (1) & (0,0) & (1,0) & (0,0,1)\\
 \nu_{\eta'}(Y)\text{ on }S'
   & () & (0) & (1) & (0,0) & (1,0) & (1,0,0).
\end{array}
\end{equation}
Indeed, if $1\in S$, the local restriction of $(1,0,1)$ wins the
left lexicographic comparison. If $1\notin S$ but $2\in S$, the
restriction of $(0,1,0)$ wins. If $S=\varnothing$, the two projected
candidates agree, so no choice of a larger full parameter is needed.
Only in the last column does the associated prefix change; there the
rank-two transition sends $(0,0,1)$ to $(1,0,0)$.

By definition, the local $\eta'$-parameter of $Y=D_{k-1}^{\eta'}$ is
$(1,0,0)$, and its outside parameter is $\gamma_j$. Restricting this
vector to the positions $S'$ in \eqref{eq:selected-local-patterns}
gives exactly the last row of \eqref{eq:exceptional-prefix-table}.
This proves \eqref{eq:prefix-minor-claim} for the exceptional minor and
completes the transport along a $3$-move.

Finally, apply \eqref{eq:prefix-minor-claim} to $\eta=\overline\beta$.
The leftmost reduced subexpression of $v$ in $\overline\beta$ is the
one already contained in $\beta$: the left-descent selection rule of
Lemma~\ref{lem:vk} has already completed $v$ by position $r$.
For $s\leq r$, the PBW construction and its triangular characterization
identify $D_s^{\overline\beta}=D_s^\beta$ and extend their parameters
by zeros in the appended positions. The associated word is constructed
from the same selected word $\beta_v$ and the same completion $\tau$.
Its finite prefix therefore computes the first $m$ coordinates of
$\mathbf a^{\dot\beta_v}$. Restricting
\eqref{eq:prefix-minor-claim} to $s\leq r$ proves the proposition.
\end{proof}

\subsection{Mutations of quivers}

We first establish a quiver calculation that will be used to keep track of
which vertices can contribute to the selected Lusztig coordinates. All
quivers in this subsection have their arrows between frozen vertices
suppressed, except that the incidence calculation in the proof of
Proposition~\ref{pro:mu1Q} temporarily retains such arrows until the final
deletion and freezing step. These arrows do not affect mutation at any
mutable vertex. A \emph{full embedding} identifies its source with an induced
subquiver of its target, with this convention for frozen vertices.

Let \(\gamma\) be a word whose first letter has color \(j\), and list its
vertices of color \(j\) in increasing word order as
\[
 x_1<x_2<\cdots<x_N.
\]
The horizontal arrows of this color are \(x_s\to x_{s+1}\).
For a vertex \(y\) of a color adjacent to \(j\), put
\[
 a(y):=\#\{s:x_s<y\},
 \qquad
 b(y):=\#\{s:x_s<y^+\},
\]
where \(b(y)=N\) if \(y^+=+\infty\). Since \(\gamma\) begins with
\(j\), one has \(a(y)\geq1\). In the next two statements, let
\(e_1,\ldots,e_N\) be the standard basis of \(\mathbb Z^N\), and set
\(e_0=0\). Encode the ordinary arrows incident with \(y\) by the vector
\[
 c(y):=\sum_{s=1}^N
 \bigl(\#\{x_s\to y\}-\#\{y\to x_s\}\bigr)e_s.
\]

\begin{lemma}\label{lem:ordinary-arrows-zigzag}
If \(a(y)=b(y)\), then \(c(y)=0\). Otherwise,
\begin{equation}\label{eq:jk-recursion}
 c(y)=
 \begin{cases}
 -e_{a(y)}+e_{b(y)},&y^+<+\infty,\\
 -e_{a(y)},&y^+=+\infty.
 \end{cases}
\end{equation}
These formulas list all ordinary arrows between color \(j\) and the
color of \(y\).
\end{lemma}

\begin{proof}
Recall that an ordinary arrow between adjacent colors is \(q\to p\)
precisely when \(p<q<p^+<q^+\). An arrow \(y\to x_s\) therefore requires
\(x_s<y<x_{s+1}<y^+\), so it occurs exactly for \(s=a(y)\) when
\(b(y)>a(y)\). An arrow \(x_s\to y\) requires
\(y<x_s<y^+<x_{s+1}\). It occurs exactly for \(s=b(y)\), provided
\(b(y)>a(y)\) and \(y^+\) is finite; here \(x_{N+1}=+\infty\).
If \(y^+=+\infty\), there is no such arrow, since the strict inequality
\(y^+<x_{s+1}\) cannot hold. This proves the assertion.
\end{proof}

\begin{figure}[htbp]
\centering
\begin{tikzpicture}[>=Stealth, every node/.style={font=\small}]
\node (a) at (0,0) {$x_a$};
\node (ap) at (2,0) {$x_{a+1}$};
\node (dots) at (4,0) {$\cdots$};
\node (b) at (6,0) {$x_b$};
\node (y) at (3,1.4) {$y$};
\draw[->] (a)--(ap);
\draw[->] (ap)--(dots);
\draw[->] (dots)--(b);
\draw[->] (b)--(y);
\draw[->] (y)--(a);
\end{tikzpicture}
\caption{The ordinary arrows attached to a vertex \(y\) with
\(a(y)<b(y)\) and finite successor.}
\label{fig:Qjk}
\end{figure}

\begin{lemma}\label{lem:firstarrow}
Suppose that an induced subquiver of a word quiver contains a vertex
\(p\) of color \(j\), every vertex of color \(j\) at or after \(p\), and
no vertex before \(p\). Whenever there is an ordinary arrow between
color \(j\) and another fixed color, the first such arrow has target
of color \(j\). Here ``first'' means minimal in the lexicographic
order on \((\max\{u,w\},\min\{u,w\})\) for the endpoints \(u,w\).
\end{lemma}

\begin{proof}
Suppose instead that the first arrow is \(s\to t\), with \(s\) of
color \(j\). Then \(p<t<s<t^+<s^+\). Let \(s'\) be the last vertex
of color \(j\) before \(t\). It exists and is retained, because
\(p\leq s'<t\) and the entire indicated color-\(j\) tail is retained.
Moreover,
\[
 s'<t<(s')^+\leq s<t^+.
\]
Thus \(t\to s'\) is an ordinary arrow of the induced subquiver. Its
endpoint pair is smaller than that of \(s\to t\), a contradiction.
\end{proof}

\subsubsection{A mutation sweep}\label{sec:mut1}

\begin{proposition}\label{pro:mu1Q}
For a word \(\gamma\) beginning with \(j\), use the vertices
\(x_1<\cdots<x_N\) above and perform the mutations
\[
 \mu_{x_1},\mu_{x_2},\ldots,\mu_{x_{N-1}}.
\]
Immediately before mutation at \(x_s\), every ordinary arrow incident
with \(x_s\) points into it, and its horizontal arrows point out of it
to \(x_{s-1}\), if \(s>1\), and to \(x_{s+1}\).

After the sweep, delete \(x_N\), freeze \(x_{N-1}\) if \(N>1\), and
suppress arrows between frozen vertices. The map
\begin{equation}\label{eq:muQ0}
 x_s\longmapsto x_{s+1}\quad(1\leq s<N),
 \qquad y\longmapsto y\quad(i_y\neq j)
\end{equation}
identifies the resulting quiver with the word quiver obtained by
removing the first letter of \(\gamma\). Before deletion, the only
possible neighbors of \(x_N\) are \(x_{N-1}\) and the last vertices
of other colors. If \(N=1\), the sweep is empty and the same statement
holds without \(x_{N-1}\).
\end{proposition}

\begin{proof}
The horizontal chain has, immediately before mutation at \(x_s\),
the arrows \(x_s\to x_{s-1}\) and \(x_s\to x_{s+1}\), with the
first omitted for \(s=1\). Reversing them establishes the same
assertion for the next step. We check the ordinary arrows at the
same time, using Lemma~\ref{lem:ordinary-arrows-zigzag}.

Consider first \(c(y)=-e_a+e_b\), where \(a<b\). Mutations before
\(x_a\) do not affect this vector. At \(x_a\), the incoming arrow
\(y\to x_a\) creates \(y\to x_{a-1}\), if \(a>1\), and
\(y\to x_{a+1}\), while the arrow at \(x_a\) is reversed. The
negative coefficient then propagates through \(x_{a+1},\ldots,x_{b-1}\).
At each step the arrow created toward the preceding vertex cancels
the positive coefficient left by the preceding mutation. At \(x_b\)
the propagated negative coefficient cancels its original positive
coefficient. Thus the final vector is
\[
 -e_{a-1}+e_{b-1}.
\]
This also covers \(b=N\), since the cancellation at \(x_N\) occurs
when \(x_{N-1}\) is mutated; mutation at \(x_N\) is unnecessary.

If \(c(y)=-e_a\), then \(y\) is the last vertex of its color and
\(a<N\). The same propagation continues to the end of the chain,
giving
\begin{equation}\label{eq_frozen_arrows}
 c_{\rm final}(y)=-e_{a-1}+e_{N-1}-e_N.
\end{equation}
The last two terms encode the boundary arrows
\(x_{N-1}\to y\to x_N\). Deleting \(x_N\) removes the latter,
and suppressing frozen--frozen arrows removes the former.
A zero initial vector remains zero.

These computations show, at every mutation, that all ordinary
arrows at the mutated vertex are incoming. Consequently no oriented
length-two path through that vertex joins two vertices whose colors
differ from \(j\). Thus no arrows among the other colors are created
or changed. They also show that no arrow between two horizontal
neighbors is created: both horizontal arrows at the mutated vertex
are outgoing. The horizontal chain after deletion has its original
orientation, with its labels shifted by one.

After removing the boundary terms in \eqref{eq_frozen_arrows}, the
ordinary incidence vectors are exactly those obtained by deleting
\(x_1\) from the word and replacing \(x_s\) by \(x_{s+1}\). This
proves the quiver identification and the assertion about the neighbors
of \(x_N\). For \(N=1\), there are no ordinary arrows incident with
\(x_1\): it is frozen and precedes all vertices of other colors.
Deleting it therefore gives the stated result directly.
\end{proof}

\subsection{Lusztig parameters of cluster variables}

Write \(m=\ell(v)\), and retain the word \(\dot\beta^M\) from the
preceding subsection. For a global-basis element, or a power of
\(q^{1/2}\) times such an element, define its selected parameter by
\begin{equation}\label{eq:selected-parameter}
 \nu(x):=\pi_m\bigl(\mathbf a^{\dot\beta^M}(x)\bigr)
 =\bigl(a_1^{\dot\beta^M}(x),\ldots,a_m^{\dot\beta^M}(x)\bigr).
\end{equation}
All parameter identities below concern this truncated vector. In
particular, a variable with \(\nu(x)=0\) need not have zero full
Lusztig parameter. Write \(E_1,\ldots,E_m\) for the standard basis
of \(\mathbb Z^m\).

\begin{lemma}\label{lem:selected-exchange}
The selected parameter is additive on quantum cluster monomials.
For an exchange relation in the full ambient seed,
\[
 x'x=q^A M_+ +q^B M_-,
\]
where \(M_+,M_-\) are quantum cluster monomials and the scalar
powers depend on the normalization, one has
\begin{equation}\label{eq:selected-exchange}
 \nu(x')+\nu(x)
 =\max_{\rm lex}\{\nu(M_+),\nu(M_-)\}.
\end{equation}
\end{lemma}

\begin{proof}
Apply Lemma~\ref{lem:exchange-parameters}, with the ambient word
\(\dot\beta^M\), and project to the first \(m\) coordinates.
Additivity on quantum cluster monomials follows from
Lemma~\ref{lem:twoprod}.
\end{proof}

\begin{definition}\label{def:Ql}
Let \(\Sigma_l:=\widetilde{\mathbf s}_l\) be the full intermediate
seed from Definition~\ref{def:svbeta}, and write \(D^l_{(j,k)}\) for
its variable at the indicated vertex. The full seed retains all
vertices throughout the mutation process. Let \(H_l\) be the set of
top vertices designated for deletion during the first \(l\) stages,
as in that definition, and set
\[
 R_l:=K\setminus H_l.
\]
Thus \(R_l\) is a set of vertices on which we will assert parameter
formulas; no seed deletion is implicit in this notation.

Extend the conventions by \(\alpha(j,0)=0\),
\(0(j)^\oplus=\min\{s\in[m]:i_{p_s}=j\}\), and
\(d_{+\infty}=+\infty\), with the minimum of an empty set taken
to be \(+\infty\). Put
\begin{align*}
 L_j(l)&:=d_{l(j)^\oplus}-\alpha(j,l),&
 U_j(l)&:=n_j-\alpha(j,l),\\
 V_l&:=\{(j,k):L_j(l)\leq k\leq U_j(l)\},&
 F_l&:=\{(j,U_j(l))\in V_l\}.
\end{align*}
An interval with lower endpoint \(+\infty\), or with lower endpoint
larger than its upper endpoint, is empty. The auxiliary quiver
\(Q^l\) is the induced subquiver of the quiver of \(\Sigma_l\) on
\(V_l\), with frozen set \(F_l\) and with frozen--frozen arrows
suppressed. Finally, write \(Z_l:=R_l\setminus V_l\).
\end{definition}

The sets \(V_l\) are nested. Indeed, to compare \(V_{l-1}\) and
\(V_l\), set
\[
i=i_{p_l},
\qquad
\alpha=\alpha(i,l-1).
\]
The color-\(i\) interval changes from
\[
[d_l-\alpha,\;n_h-\alpha]
\quad\text{to}\quad
[d_{l^\oplus}-\alpha-1,\;n_h-\alpha-1].
\]
If \(l^\oplus<+\infty\), then \(d_{l^\oplus}\geq d_l+1\), so the
lower endpoint does not decrease, while the upper endpoint decreases
by one. If \(l^\oplus=+\infty\), the new interval is empty. The
intervals of all other colors remain unchanged. Hence
\[
V_l\subseteq V_{l-1}.
\]

For \(l<m\), put \(h=i_{p_{l+1}}\). Since
\[
l(h)^\oplus=l+1,
\qquad
\alpha(h,l)=a_{l+1}-1,
\]
the color-\(h\) vertices of \(V_l\setminus F_l\) are precisely
\[
(h,b_{l+1}+1),\ldots,(h,n_h-a_{l+1}),
\]
namely the mutation vertices of \(\widetilde{\mu}_{l+1}\).

Consequently, vertices outside \(V_l\) are never mutated later. To
see the same for \(F_l\), fix
\[
x=(j,U_j(l))\in F_l
\]
and let \(s=l(j)^\oplus\). If \(s=+\infty\), no later mutation
sequence has color \(j\). Otherwise, \(x\) remains the terminal
\(j\)-vertex until stage \(s\), so it is excluded from
\(\widetilde{\mu}_s\); after that stage, the color-\(j\) upper
endpoint decreases by one, and hence \(x\notin V_s\). Thus no vertex
of \(F_l\) is mutated at a later stage.

For comparison with the original word, define
\begin{equation}\label{eq:support-image}
 S_l:=\{(j,t):d_{l(j)^\oplus}\leq t\leq n_j\},
 \qquad
 \Phi_l(j,k):=(j,k+\alpha(j,l)).
\end{equation}
Thus \(\Phi_l\) is a bijection from \(V_l\) to \(S_l\). Word order
on \(Q^l\) will mean the order transported by \(\Phi_l\). The set
\(S_l\) consists of a terminal interval for each color. The quiver
induced by \(S_l\) in \(Q_\beta\) is the word quiver of the word
obtained by retaining just these occurrences: deleting an initial
segment of a color does not change the successor of any retained
occurrence. If \(l<m\), the first letter of this retained word is
\(i_{p_{l+1}}\). Consequently it is a full subquiver of
\(Q_{\geq p_{l+1}}\), and it contains the entire color-
\(i_{p_{l+1}}\) tail beginning at \(p_{l+1}\).

\begin{theorem}\label{thm:induction}
For every \(0\leq l\leq m\), the following assertions hold.
\begin{enumerate}
\item For every \((j,k)\in R_l\) and \(s\in[m]\),
\begin{equation}\label{eq:blfacasel}
 a_s^{\dot\beta^M}\bigl(D^l_{(j,k)}\bigr)
 =
 \begin{cases}
 1,&p_s=(j,t),\quad d_{l(j)^\oplus}\leq t\leq k+\alpha(j,l),\\
 0,&\text{otherwise}.
 \end{cases}
\end{equation}
In particular, \(Z_l\) is precisely the set of vertices in \(R_l\)
whose variables have zero selected parameter.
\item The map \(\Phi_l\) identifies \(Q^l\) with the full subquiver
of \(Q_\beta\) on \(S_l\). For \(l<m\), it therefore gives a full
embedding
\[
 \Phi_l:Q^l\hookrightarrow Q_{\geq p_{l+1}}.
\]
Whenever there is an ordinary arrow between color \(i_{p_{l+1}}\)
and another color, the first such arrow in \(Q^l\) has target of
color \(i_{p_{l+1}}\).
\item In the full quiver of \(\Sigma_l\), no vertex of \(H_l\) is
adjacent to a vertex of \(V_l\setminus F_l\).
\item One has \(Q^m=\varnothing\), and every variable of the final
seed \(\mathbf s(v,\beta)\) has zero selected parameter.
\end{enumerate}
\end{theorem}

\begin{proof}
If \(m=0\), then \(v=e\), the vector \(\nu\) is empty, and all the
sets \(V_l\) and \(S_l\) under consideration are empty. The assertions
are immediate. Assume henceforth that \(m>0\).

We prove the first three assertions simultaneously. At \(l=0\),
Proposition~\ref{pro:bfadotbeta} and the definition of \(D_k^\beta\)
give
\[
 \nu(D^0_{(j,k)})
 =\sum_{\substack{s\in[m]:\ i_{p_s}=j\\d_s\leq k}}E_{s}.
\]
This is \eqref{eq:blfacasel}. Hence its nonzero support is exactly
\(V_0\), and the quiver identification is the identity on \(S_0\).
In particular, the initial comparison is the full embedding
\begin{equation}\label{eq:Q0}
 Q^0\hookrightarrow Q_{\geq p_1},
 \qquad (Q^0)_0=S_0;
\end{equation}
it need not be an equality of the two quivers. The third assertion is
vacuous because \(H_0=\varnothing\).

Suppose that the assertions hold at \(l-1\), and set
\[
 j=i_{p_l},\qquad
 \alpha=\alpha(j,l-1),\qquad
 L=d_l-\alpha=b_l+1,\qquad U=n_j-\alpha.
\]
The color-\(j\) vertices of \(V_{l-1}\) are
\((j,L),\ldots,(j,U)\), and the mutation sequence is
\[
 \mu_{(j,L)},\mu_{(j,L+1)},\ldots,\mu_{(j,U-1)}.
\]
It is empty when \(L=U\). By the quiver identification at \(l-1\),
Proposition~\ref{pro:mu1Q} applies to this sequence on \(Q^{l-1}\).
Restriction to an induced subquiver commutes with mutation at a
retained vertex, since the formula for an entry between two retained
vertices uses only their entries with the mutated vertex. Thus this
calculation gives the induced quiver of the full mutation process on
\(V_{l-1}\).

By the third induction assertion, no vertex of \(H_{l-1}\) is
adjacent to any vertex being mutated. These incidences remain zero
throughout the sweep: mutating at a vertex with zero incidence to
\(H_{l-1}\) cannot create an incidence with \(H_{l-1}\).
Every variable attached to \(Z_{l-1}\) has zero selected parameter
and is unchanged, since these vertices are never mutated.

We now compute the selected parameters without discarding any
factors from the full exchange relation. Suppose first that \(L<U\).
Immediately before mutation at \((j,k)\), where \(L\leq k<U\),
Proposition~\ref{pro:mu1Q} separates the retained neighbors into
outgoing horizontal neighbors and incoming ordinary neighbors.
Consequently the full exchange relation can be written as
\begin{equation}\label{eq:induction-exchange-general}
 D^{\rm new}_{(j,k)}D^{\rm old}_{(j,k)}
 =q^{A_k}U_{+,k}H_k+q^{B_k}U_{-,k}O_k.
\end{equation}
Here \(H_k\) is the quantum monomial in the retained horizontal
neighbors, \(O_k\) is the quantum monomial in the retained ordinary
neighbors, and \(U_{\pm,k}\) collect the factors at vertices of
\(Z_{l-1}\). Products are taken in a fixed order and scalar powers
are absorbed into \(A_k,B_k\in\tfrac12\mathbb Z\). There are no
factors from \(H_{l-1}\), by the preceding paragraph, and
\[
 \nu(U_{+,k})=\nu(U_{-,k})=0.
\]
In particular, at the first mutation the horizontal monomial is
\begin{equation}\label{eq:induction-exchange-first}
 H_L=D^{l-1}_{(j,L+1)}.
\end{equation}
A predecessor outside \(V_{l-1}\), if present in the full seed, belongs
to one of the factors \(U_{\pm,L}\); it is not inserted into
\(H_L\). For the later mutations, additivity gives
\[
 \nu(H_k)=
 \nu(D^{l-1}_{(j,k+1)})+\nu(D^l_{(j,k-1)})
 \qquad(k>L).
\]

The induction hypothesis gives
\begin{equation}\label{eq:Dl-1=0}
 a_s^{\dot\beta^M}(D^{l-1}_u)=0
 \quad(s<l,\ u\in R_{l-1}).
\end{equation}
Every ordinary neighbor has color different from \(j\), so the
\(l\)-th coordinate of \(\nu(O_k)\) is also zero. The old variables
\(D^{l-1}_{(j,k)}\) and \(D^{l-1}_{(j,k+1)}\) have \(l\)-th
coordinate one. At the first mutation, \(\nu(H_L)\) therefore has
its first \(l-1\) coordinates zero and its \(l\)-th coordinate one.
Lemma~\ref{lem:selected-exchange} selects this monomial and gives
\begin{equation}\label{eq:induction-recurrence}
 \nu(D^l_{(j,L)})
 =\nu(D^{l-1}_{(j,L+1)})-\nu(D^{l-1}_{(j,L)})
 =\nu(D^{l-1}_{(j,L+1)})-E_l.
\end{equation}
The last equality uses \(L+\alpha=d_l\), so that
\(\nu(D^{l-1}_{(j,L)})=E_l\).

Inductively, suppose that the preceding new variable has selected
parameter \(\nu(D^{l-1}_{(j,k)})-E_l\). Its first \(l\)
coordinates vanish. Thus \(\nu(H_k)\) again has first \(l-1\)
coordinates zero and \(l\)-th coordinate one, whereas
\(\nu(O_k)=0\) in its first \(l\) coordinates. Applying
Lemma~\ref{lem:selected-exchange} yields
\begin{align}
 \nu(D^l_{(j,k)})
 &=\nu(D^{l-1}_{(j,k+1)})+\nu(D^l_{(j,k-1)})
       -\nu(D^{l-1}_{(j,k)})\notag\\
 &=\nu(D^{l-1}_{(j,k+1)})-E_l.
 \label{eq:induction-recurrence-general}
\end{align}
This proves the recurrence through the entire sweep. The right-hand
side is precisely the sum of those \(E_s\) with
\(s>l\), \(i_{p_s}=j\), and \(d_s\leq k+\alpha+1\), which is
\eqref{eq:blfacasel} at stage \(l\).

The variables of all other colors are unchanged and their thresholds
are unchanged. The old vertices in \(Z_{l-1}\) remain zero, and their
thresholds can only increase. The sole new member of \(H_l\) is
\((j,U)\), so no assertion about its selected parameter is needed.
This proves the first assertion for every vertex of \(R_l\).
If \(L=U\), the color-\(j\) part has just this one vertex, which is
moved into \(H_l\); there is no surviving color-\(j\) variable to
compute, and the same argument for the other vertices proves the
assertion in the empty-sweep case.

For the second assertion, first apply the quiver identification in
Proposition~\ref{pro:mu1Q}. Removing \((j,U)\) and shifting the
color-\(j\) labels by one identifies the induced quiver with the word
quiver obtained by deleting its first letter. The additional vertices
outside \(V_l\) are exactly those below the next selected occurrence
of color \(j\); by the parameter formula just proved their selected
parameters are zero. Deleting these vertices amounts to deleting a
further initial segment of that color. All other color thresholds
are unchanged. The resulting word is therefore exactly the word
retained by \(S_l\), and the resulting map is \(\Phi_l\).
Terminal occurrences of retained colors are preserved, so the stated
frozen sets agree. This proves the full quiver identification.
For \(l<m\), the retained color-\(i_{p_{l+1}}\) tail is complete;
Lemma~\ref{lem:firstarrow}, with its explicit retained-tail hypothesis,
proves the assertion about the first ordinary arrow.

It remains to check the third assertion. Old vertices of \(H_{l-1}\)
remain nonadjacent to all old nonterminal vertices of \(V_{l-1}\),
as already observed. The set \(V_l\setminus F_l\) is contained in
that set of old nonterminal vertices. The new vertex \((j,U)\) has,
within the retained quiver, only the new terminal vertex \((j,U-1)\)
and old terminal vertices of other colors as possible neighbors,
by Proposition~\ref{pro:mu1Q}. Deleting zero-parameter vertices does
not add any arrows. Thus \((j,U)\) is also nonadjacent to
\(V_l\setminus F_l\). This completes the simultaneous induction.

Finally, for \(l=m\), every next selected occurrence is \(+\infty\).
Thus \(V_m=S_m=\varnothing\), giving \(Q^m=\varnothing\), and
\eqref{eq:blfacasel} gives \(\nu(D^m_u)=0\) for every \(u\in R_m\).
The vertices of the final seed \(\mathbf s(v,\beta)\) are precisely
\(R_m\), and its variables are the corresponding variables of the
full intermediate seed. This proves the final assertion.
\end{proof}

\begin{remark}
The recurrence \eqref{eq:induction-recurrence-general} removes the
contribution \(E_l\) of the current selected letter and shifts the
upper color-occurrence bound by one. It concerns only the first
\(m\) coordinates in the fixed \(\dot\beta^M\)-parameterization.
Neither the full \(\dot\beta^M\)-parameter nor the parameter for the
original word \(\beta\) is asserted to obey this recurrence.
\end{remark}

\begin{example}\label{ex:rank-two-truncated-mutation}
Consider type $A_2$, with $\beta=(1,2,1,1)$ and $v=s_2s_1$.
The selected positions are $p_1=2$ and $p_2=3$; hence
$\widetilde\mu_1=\mathrm{Id}$ and $\widetilde\mu_2=\mu_3$ in positional
labels. Proposition~\ref{pro:bfadotbeta} gives
\[
\begin{array}{c|cccc}
 k&1&2&3&4\\\hline
 \nu(D_k^\beta)&(0,0)&(1,0)&(0,1)&(0,1).
\end{array}
\]
Write $x_k$ for the commutative specialization of $D_k^\beta$.
The full quiver has the arrows $1\to3\to4$ and $2\to1$.
Consequently, its exchange relation at vertex $3$ is
\[
 x_3'x_3=x_1+x_4.
\]
The variable $x_1$ is absent from the auxiliary support quiver but
remains in this full exchange relation. Applying the selected-parameter
formula to its quantum lift yields
\[
 \nu(D_3')=\max_{\rm lex}\{(0,0),(0,1)\}-(0,1)=(0,0).
\]
The designated deletion set is $H_2=\{2,4\}$. The surviving vertices
$1,3$ are both frozen, and therefore
\[
 \cA_0(\mathbf s(v,\beta))=\mathbb Z[x_1,x_3'],
 \qquad
 \cA_0(\mathbf s(v,\beta))_{\rm loc}
 =\mathbb Z[x_1^{\pm1},(x_3')^{\pm1}].
\]
Theorem~\ref{thm:categorification} realizes every monomial
$x_1^a(x_3')^b$, with $a,b\geq0$, by a real simple object in
$\mathscr C_{v,\beta}$.
\end{example}

\begin{example}\label{ex:positive-rank-truncated-mutation}
Consider type $A_2$, with $\beta=(1,2,1,2)$ and $v=s_1$.
Use positional labels and abbreviate $D_k=D_k^\beta$.
Here $p_1=1$, $\widetilde\mu_1=\mu_1$, and $H_1=\{3\}$.
Identifying the one-coordinate selected parameter with its entry,
Proposition~\ref{pro:bfadotbeta} gives
\[
 \bigl(\nu(D_1),\nu(D_2),\nu(D_3),\nu(D_4)\bigr)=(1,0,1,0).
\]
The initial frozen vertices are $3,4$, and the full initial quiver has
the arrows
\[
 1\longrightarrow3\longrightarrow2\longrightarrow1,
 \qquad 2\longrightarrow4.
\]
Thus, writing $D_1'=\mu_1(D_1)$, the full quantum exchange relation is
\[
 D_1'D_1=q^A D_2+q^B D_3,
 \qquad A,B\in\tfrac12\mathbb Z.
\]
In particular, the factor $D_2$, whose selected parameter is zero,
remains in the exchange relation. Lemma~\ref{lem:selected-exchange}
gives
\[
 \nu(D_1')=\max\{0,1\}-1=0.
\]
Mutation at $1$ creates the arrow $2\to3$, which cancels $3\to2$,
and reverses the two arrows incident with $1$. Consequently the full
quiver after the sweep is
\[
 3\longrightarrow1\longrightarrow2\longrightarrow4.
\]
Delete $3$ and freeze its sole remaining neighbor $1$; vertex $4$
remains frozen. Hence
\[
 J=\{1,2,4\},\qquad J_{\rm fr}=\{1,4\},\qquad J_{\rm ex}=\{2\}.
\]
The final seed has exchangeable rank one, with quiver
\begin{center}
\begin{tikzpicture}[>=Stealth, every node/.style={font=\small}]
\node[draw,rectangle] (one) at (0,0) {$1$};
\node[draw,circle] (two) at (1.5,0) {$2$};
\node[draw,rectangle] (four) at (3,0) {$4$};
\draw[->] (one)--(two);
\draw[->] (two)--(four);
\end{tikzpicture}
\end{center}
Here boxes indicate frozen vertices, and the attached variables are
$D_1',D_2,D_4$, all with zero selected parameter. Mutating at the
retained exchangeable vertex $2$ gives
\[
 D_2'D_2=q^{A'}D_1'+q^{B'}D_4,
 \qquad A',B'\in\tfrac12\mathbb Z.
\]
This is also the exchange relation in the full ambient seed, since
the deleted vertex $3$ is not adjacent to $2$. Its commutative
specialization is $x_2'x_2=x_1'+x_4$, where $x$ denotes specialization
of the corresponding $D$. Applying Lemma~\ref{lem:selected-exchange}
once more yields
\[
 \nu(D_2')=\max\{\nu(D_1'),\nu(D_4)\}-\nu(D_2)
 =\max\{0,0\}-0=0.
\]
\end{example}

\section{Categorification of twisted products of flag varieties}
\label{sec:category}

Throughout this section, the Cartan matrix is of type \(ADE\).
We use the integral cluster algebra \(\cA_0\) of
Section~\ref{sec:pre}; in particular, frozen variables are not inverted.

\subsection{Hernandez--Leclerc's category and complete duality data}

Let \(\fg\) be the simple Lie algebra associated with \(C\). Fix
\(\zeta\in\CC^*\) which is not a root of unity, and consider the
quantum affine algebra \(U_{\zeta}'(\widehat{\fg})\).
Its parameter \(\zeta\) is distinct from the formal bosonic and cluster
parameters \(q\) and \(t\). For \((i,a)\in I\times\CC^*\), write
\(V_i(a)\) for the fundamental module. Fix a height function
\(\xi:I\to\ZZ\) with \(|\xi(i)-\xi(j)|=1\) whenever \(c_{ij}=-1\),
and set
\[
\Delta^\xi=\{(i,p)\in I\times\ZZ\mid p-\xi(i)\in2\ZZ\}.
\]
The Hernandez--Leclerc category \(\mathscr C^{\ZZ}\) is the full
subcategory of finite-dimensional type~\(1\)
\(U_{\zeta}'(\widehat{\fg})\)-modules generated by
\[
\{V_i(\zeta^p)\mid(i,p)\in\Delta^\xi\}\cup\{\mathbf1\}
\]
under tensor products, subquotients, and extensions.

Write \(\mathcal D\) and \(\mathcal D^{-1}\) for the right and left
dual functors. For simple modules \(M,N\), put
\(M\triangledown N:=\operatorname{hd}(M\otimes N)\).
A simple module \(L\) is \emph{real} if \(L\otimes L\) is simple.
We use the \(R\)-matrix invariants of
\cite[Section~1.3]{kashiwara2025monoidal}, with the normalization
\[
\mathfrak d(M,N)=\frac12\bigl(\Lambda(M,N)+\Lambda(N,M)\bigr).
\]
A real simple module \(L\) is a \emph{root module} if
\[
\mathfrak d(L,\mathcal D^kL)=\delta_{k,1}+\delta_{k,-1}
\qquad(k\in\ZZ).
\]

\begin{definition}[{\cite[Definitions~2.6 and~2.10]{kashiwara2025monoidal}}]
\label{def:complete-duality-datum}
A family of root modules \(\DD=\{L_i^{\DD}\}_{i\in I}\) is a
\emph{strong duality datum} if
\[
\mathfrak d(L_i^{\DD},\mathcal D^kL_j^{\DD})
=-\delta_{k,0}c_{ij}
\qquad(i\ne j,\ k\in\ZZ).
\]
Let \(\mathscr C_{\DD}\) be the full subcategory generated by
\(\{L_i^{\DD}\}_{i\in I}\cup\{\mathbf1\}\) under tensor products,
subquotients, and extensions. The datum \(\DD\) is \emph{complete} if
every simple \(M\in\mathscr C^{\ZZ}\) has a presentation
\[
M\simeq\operatorname{hd}\bigl(
\cdots\otimes\mathcal D^2M_2\otimes\mathcal DM_1
\otimes M_0\otimes\mathcal D^{-1}M_{-1}\otimes\cdots\bigr),
\]
where each \(M_k\) is simple in \(\mathscr C_{\DD}\) and
\(M_k\simeq\mathbf1\) for all but finitely many \(k\).
\end{definition}

For \([a,c]\subset\ZZ\), denote by \(\mathscr C^{\DD}[a,c]\) the full
subcategory generated by \(\{\mathcal D^mL_i^{\DD}\mid
i\in I,\ a\le m\le c\}\cup\{\mathbf1\}\) under the same operations.
We use the same notation for any complete datum and set
\(\mathscr C^{\DD}[0,\infty):=\bigcup_{c\ge0}\mathscr C^{\DD}[0,c]\).
Thus the interval indices here specify powers of the right dual functor.
From now on, \(\DD\) is an arbitrary fixed complete duality datum.

For \(j\in I\), define \(\mathscr S_j(\DD)=\{L_i'\}_{i\in I}\) by
\[
L_i'=
\begin{cases}
\mathcal D(L_j^{\DD}),&i=j,\\
L_j^{\DD}\triangledown L_i^{\DD},&c_{ij}=-1,\\
L_i^{\DD},&\text{otherwise}.
\end{cases}
\]
This is again a complete duality datum
\cite[Proposition~2.11]{kashiwara2025monoidal}.
For a finite word \(\beta=(i_1,\ldots,i_r)\), define recursively
\begin{equation}\label{eq:categorical-cuspidals}
\DD^{(0)}:=\DD,\qquad
\DD^{(k)}:=\mathscr S_{i_k}(\DD^{(k-1)}),\qquad
C_k^{\DD,\beta}:=L_{i_k}^{\DD^{(k-1)}}.
\end{equation}
This agrees with the affine cuspidal modules in
\cite[Corollary~3.23 and~(4.4)]{kashiwara2025monoidal}.
The same recursion defines \(C_k^{\DD,\omega}\) for a one-sided
infinite word \(\omega\).

\begin{definition}
For a finite or one-sided infinite word \(\omega\), let
\(\mathscr C_{\DD}(\omega)\) be the smallest full subcategory of
\(\mathscr C^{\ZZ}\) containing its affine cuspidal modules and
\(\mathbf1\), and closed under tensor products, subquotients, and
extensions.
\end{definition}

If \(\beta\) represents \(b\in\Br^+\), the category
\(\mathscr C_{\DD}(\beta)\) depends only on \(b\), by
\cite[Corollary~5.34]{kashiwara2025monoidal}. We also write
\(\mathscr C_{\DD}(b)\), or \(\mathscr C(b)\) when the datum is fixed.

Let \(R:=\ZZ[q^{\pm1/2}]\), and denote by
\(\widehat{\cA}_R\) the integral form spanned by the global basis, and set
\(\widehat{\cA}_R(b):=\widehat{\cA}_R\cap\widehat{\cA}(b)\).
The homomorphism of
\cite[Proposition~3.18 and Theorem~3.19]{kashiwara2025monoidal} is
\begin{equation}\label{eq:categorical-specialization}
\Phi_{\DD}:\widehat{\cA}_R\longrightarrow K_0(\mathscr C^{\ZZ}),
\qquad q^{1/2}\longmapsto1.
\end{equation}
It induces an isomorphism after quotienting by \(q^{1/2}-1\), and
\[
\Phi_{\DD}(P_k^\beta)=[C_k^{\DD,\beta}].
\]
Consequently, for \(\mathbf a\in\NN^r\), the PBW monomial specializes
to the class of the standard module
\begin{equation}\label{eq:categorical-standard-module}
\mathsf P^{\DD,\beta}(\mathbf a):=
(C_r^{\DD,\beta})^{\otimes a_r}\otimes\cdots\otimes
(C_1^{\DD,\beta})^{\otimes a_1},
\qquad
\Phi_{\DD}(P^\beta(\mathbf a))=[\mathsf P^{\DD,\beta}(\mathbf a)].
\end{equation}
Here all scalar powers used to normalize PBW monomials specialize to one.

\begin{theorem}[{\cite[Theorems~9.4 and~9.7, Corollary~9.8]
{kashiwara2025monoidal}}]\label{theo:isoAb}
For \(b\in\Br^+\) and an expression \(\beta\) of \(b\), there is an
isomorphism of \(\ZZ\)-algebras
\[
\cA_0(\mathbf s(\beta))\xrightarrow{\sim}K_0(\mathscr C_{\DD}(b))
\]
realizing every cluster monomial, with nonnegative frozen exponents,
as the class of a real simple module.
Moreover, let \(X\) be the normalized quantum lift of such an ambient
cluster monomial in \(\widehat{\cA}_R(b)\). Then \(X\) is a normalized
global basis element, and the corresponding simple module \(M_X^{\DD}\)
satisfies
\begin{equation}\label{eq:cluster-monomial-quantizability}
\Phi_{\DD}(X)=[M_X^{\DD}].
\end{equation}
The element \(X\) is independent of the complete datum \(\DD\).
\end{theorem}

In this statement, the quantum torus parameter \(t\) is related to
the bosonic parameter by \(t^{1/2}\mapsto q^{-1/2}\), as in
\cite[Theorem~9.7]{kashiwara2025monoidal}. The identification of the
initial variables with affine determinantial modules is supplied by
\cite[Theorem~6.10]{kashiwara2025monoidal}. Thus the statement applies
to the normalized variables \(D_k^\beta\) used in
Section~\ref{sec:bosonic}. We will use only
\eqref{eq:cluster-monomial-quantizability}, and do not require an
identification of the quantum Grothendieck ring for an arbitrary datum.

\begin{theorem}[{\cite[Lemma~5.26, Theorem~5.28, and Proposition~5.33]
{kashiwara2025monoidal}}]\label{theo:Lusztigsim}
Let \(\beta=(i_1,\ldots,i_r)\). The module
\[
\mathsf V^{\DD,\beta}(\mathbf a):=
\operatorname{hd}\mathsf P^{\DD,\beta}(\mathbf a)
\]
is simple for every \(\mathbf a\in\NN^r\), and these modules give a
bijection between \(\NN^r\) and the simple objects of
\(\mathscr C_{\DD}(\beta)\), up to isomorphism. We write
\(\bfa^\beta(M)=\mathbf a\) when
\(M\simeq\mathsf V^{\DD,\beta}(\mathbf a)\).
Furthermore,
\begin{equation}\label{eq:categorical-PBW-triangularity}
[\mathsf P^{\DD,\beta}(\mathbf a)]
=[\mathsf V^{\DD,\beta}(\mathbf a)]
+\sum_{\mathbf b\prec\mathbf a}n_{\mathbf a,\mathbf b}
[\mathsf V^{\DD,\beta}(\mathbf b)],
\qquad n_{\mathbf a,\mathbf b}\in\NN.
\end{equation}
The same assertions hold for a one-sided infinite word, with parameters
of finite support.
\end{theorem}

We need the compatibility between these categorical parameters and the
parameters of those global basis elements that represent simple modules.

\begin{lemma}\label{lem:categorical-PBW-compatibility}
Let \(\omega\) be a finite word, let
\(X=\widetilde G^\omega(\mathbf a)\in\widehat{\cA}_R(\omega)\), and
suppose that \(\Phi_{\DD}(X)=[M]\) for a simple module
\(M\in\mathscr C_{\DD}(\omega)\). Then
\[
\bfa^\omega(M)=\mathbf a=\bfa^\omega(X).
\]
This assertion also holds for a one-sided infinite word whenever the
element and the module belong to the subalgebra and subcategory of a
finite prefix.
\end{lemma}

\begin{proof}
The normalization relating \(\widetilde G^\omega(\mathbf a)\) and
\(G^\omega(\mathbf a)\) is a power of \(q^{1/2}\), so both have the
same image under \(\Phi_{\DD}\). Specializing the triangular
PBW expansion of \(G^\omega(\mathbf a)\) gives
\[
[M]=[\mathsf P^{\DD,\omega}(\mathbf a)]
+\sum_{\mathbf b\prec\mathbf a}c_{\mathbf b}
[\mathsf P^{\DD,\omega}(\mathbf b)],
\qquad c_{\mathbf b}\in\ZZ.
\]
Apply \eqref{eq:categorical-PBW-triangularity} to each term. The
coefficient of \([\mathsf V^{\DD,\omega}(\mathbf a)]\) is one, and
all remaining simple classes have strictly smaller parameters.
Since simple classes form a \(\ZZ\)-basis of the Grothendieck group,
the equality with the single simple class \([M]\) implies
\(M\simeq\mathsf V^{\DD,\omega}(\mathbf a)\).
For an infinite word, choose a finite prefix containing the parameters
in question. Enlarging the prefix appends zero coordinates and leaves
the standard module and its head unchanged; uniqueness in
Theorem~\ref{theo:Lusztigsim} finishes the proof.
\end{proof}

\subsection{The subcategory determined by the vanishing parameters}

Fix \(b\in\Br^+\), an expression \(\beta\) of \(b\), and
\(v\le\delta(b)\). Put \(m:=\ell(v)\), and retain the infinite word
\(\dot\beta_v\) of Section~\ref{sec:algebracA}. Define
\(\mathscr C^v\) to be the full subcategory generated by
\[
\{C_k^{\DD,\dot\beta_v}\mid k>m\}\cup\{\mathbf1\}
\]
under tensor products, subquotients, and extensions, and set
\[
\mathscr C_{v,\beta}:=\mathscr C_{\DD}(\beta)\cap\mathscr C^v.
\]
All these are Serre monoidal subcategories. In particular, their
Grothendieck groups embed into \(K_0(\mathscr C^{\ZZ})\) via the
corresponding subsets of the basis of simple classes.

\begin{proposition}\label{prop:categorical-tail-intrinsic}
Let \(\nu=(j_1,\ldots,j_m)\) be a reduced expression of \(v\). Set
\[
\mathbb E_0:=\DD,\qquad
\mathbb E_s:=\mathscr S_{j_s}(\mathbb E_{s-1})\quad(1\le s\le m),
\qquad
\DD_v:=\mathbb E_m
=\mathscr S_{j_m}\cdots\mathscr S_{j_1}(\DD).
\]
The complete datum \(\DD_v\) depends only on \(v\), up to isomorphism,
and
\begin{equation}\label{eq:categorical-tail-intrinsic}
\mathscr C^v=\mathscr C^{\DD_v}[0,\infty).
\end{equation}
In particular, \(\mathscr C^v\) is independent of the reduced expression
of \(v\) and of its completion to a reduced expression of \(w_0\).
Consequently, \(\mathscr C_{v,\beta}\) depends only on \(\DD,v,b\).
\end{proposition}

\begin{proof}
The braid relations for the operators \(\mathscr S_i\), proved in
\cite[Corollary~3.22]{kashiwara2025monoidal}, imply that \(\DD_v\)
is independent of \(\nu\). The order of the operators is fixed by the
recursion; with this convention,
\[
\Phi_{\DD_v}=\Phi_{\DD}\circ T_v
\]
by \cite[Corollary~3.23]{kashiwara2025monoidal}.

For any complete datum \(\mathbb E\) and \(n\ge1\), the periodic
word formed from a reduced word of \(w_0\) and its successive
\(*\)-transforms gives
\[
\mathscr C_{\mathbb E}(\Delta^n)=\mathscr C^{\mathbb E}[0,n-1].
\]
Indeed, its first block consists of Schur--Weyl images of finite-type
cuspidal modules, all lying in \(\mathscr C_{\mathbb E}\), and includes
every \(L_i^{\mathbb E}\). Its block indexed by \(p\ge0\) consists of
the \(\mathcal D^p\)-duals of the first block; see
\cite[(3.24) and Theorem~4.6(i)--(ii)]{kashiwara2025monoidal}.
Since duality is exact and reverses tensor products, these first \(n\)
blocks generate exactly the subcategory generated by
\(\mathcal D^pL_i^{\mathbb E}\) for \(0\le p<n\).

To compute \(\mathscr C^v\), take \(\nu=\beta_v\), write its chosen
completion as \(\overline w_0=\nu\tau\), and let \(d\in\Br^+\) be
represented by \(\tau\). Removing the initial
\(m\) letters from \(\dot\beta_v\) restarts
\eqref{eq:categorical-cuspidals} with datum \(\DD_v\). The finite
prefixes ending at the subsequent block boundaries represent
\(d\Delta^n\), \(n\ge0\). Hence
\[
\mathscr C^v=\bigcup_{n\ge0}\mathscr C_{\DD_v}(d\Delta^n).
\]
Let \(\theta\) be the braid automorphism \(\sigma_i\mapsto\sigma_{i^*}\).
The Garside identity gives \(d\Delta^n=\Delta^n\theta^n(d)\).
If \(u\) is represented by \(\nu\), then \(ud=\Delta\) and
\(u\Delta=\Delta\theta(u)\) imply \(d\theta(u)=\Delta\).
Thus \(\theta^n(d)\) is a positive braid prefix of \(\Delta\), and
\(\Delta^n\) is a prefix of \(d\Delta^n\), which is itself a prefix
of \(\Delta^{n+1}\). Prefix inclusion and independence
of finite expressions
\cite[Corollary~5.34]{kashiwara2025monoidal} therefore give
\[
\mathscr C_{\DD_v}(\Delta^n)
\subset\mathscr C_{\DD_v}(d\Delta^n)
\subset\mathscr C_{\DD_v}(\Delta^{n+1}).
\]
Taking unions and using the block description proves
\eqref{eq:categorical-tail-intrinsic}. The final assertion also uses
the independence of \(\mathscr C_{\DD}(\beta)\) from the expression
of \(b\).
\end{proof}

Choose \(M\) such that \(b\) is a positive braid prefix of
\(\Delta^{M+1}\), and let \(\dot\beta^M\) be the corresponding finite
prefix of \(\dot\beta_v\). The independence of the expression gives
\[
\mathscr C_{\DD}(\beta)\subset
\mathscr C_{\DD}(\Delta^{M+1})
=\mathscr C_{\DD}(\dot\beta^M)
\subset\mathscr C_{\DD}(\dot\beta_v).
\]
Thus all the parameters used below are defined in a common category.

\begin{lemma}\label{lem:categorical-prefix-vanishing}
For a simple object \(S\in\mathscr C_{\DD}(\dot\beta_v)\), one has
\[
S\in\mathscr C^v
\quad\Longleftrightarrow\quad
\bfa^{\dot\beta_v}_i(S)=0\quad(1\le i\le m).
\]
Consequently, if \(X\) is an ambient normalized quantum cluster
monomial in \(\widehat{\cA}_R(b)\), then
\[
M_X^{\DD}\in\mathscr C_{v,\beta}
\quad\Longleftrightarrow\quad
\bfa^{\dot\beta_v}_i(X)=0\quad(1\le i\le m).
\]
\end{lemma}

\begin{proof}
Suppose first that \(\bfa^{\dot\beta_v}_i(S)=0\) for \(i\le m\).
The standard module with parameter \(\bfa^{\dot\beta_v}(S)\) uses
only cuspidal modules indexed by \(k>m\). Its head is \(S\), so
\(S\in\mathscr C^v\).

Conversely, a simple object of \(\mathscr C^v\) occurs as a
composition factor of a finite tensor product of cuspidal modules
indexed by \(k>m\). This follows directly from the definition of the
subcategory and exactness of tensor products. The Grothendieck ring
is commutative, so reordering these tensor factors does not change
their composition multiplicities. Hence \(S\) occurs in
\(\mathsf P^{\DD,\dot\beta_v}(\mathbf a)\) for a finite-support
parameter with \(a_i=0\) for \(i\le m\). By
\eqref{eq:categorical-PBW-triangularity}, its parameter \(\mathbf b\)
satisfies \(\mathbf b\preceq\mathbf a\). If \(\mathbf b\ne\mathbf a\),
the first coordinate at which they differ is smaller for
\(\mathbf b\). It cannot lie in \([m]\), because those coordinates of
\(\mathbf a\) are zero and \(\mathbf b\) is nonnegative. Therefore
\(b_i=0\) for \(i\le m\).
The final assertion follows from
\eqref{eq:cluster-monomial-quantizability} and
Lemma~\ref{lem:categorical-PBW-compatibility}, applied in the finite
prefix \(\dot\beta^M\).
\end{proof}

\begin{theorem}\label{thm:categorification}
Let \(\DD\) be a complete duality datum, let \(b\in\Br^+\), let
\(\beta\) be an expression of \(b\), and let \(v\le\delta(b)\).
There is an injective homomorphism of \(\ZZ\)-algebras
\[
\cA_0(\mathbf s(v,\beta))\hookrightarrow
K_0(\mathscr C_{v,\beta}).
\]
Under this homomorphism, every cluster monomial with nonnegative
frozen exponents is the class of a real simple object of
\(\mathscr C_{v,\beta}\).
\end{theorem}

\begin{proof}
We first justify the use of the ambient categorification. In
Definition~\ref{def:svbeta}, the seed \(\mathbf s(v,\beta)\) is
obtained from an ambient mutated seed by freezing vertices and
deleting the vertices \(J_m\). If \(J_{\rm ex}\) is the final set of
exchangeable vertices, the construction gives
\[
B_{J_m\times J_{\rm ex}}=0.
\]
This condition persists under mutations in \(J_{\rm ex}\): for a
deleted row \(d\) and a retained mutable column \(j\), the mutation
formula at \(k\in J_{\rm ex}\) uses only the entries \(b_{dj}\) and
\(b_{dk}\), both of which are zero. Thus the exchange relations on
the retained variables coincide with their ambient exchange
relations after any such mutation sequence. The same holds for
their quantum lifts, since the compatible form is restricted from
the ambient torus. Every seed, and every cluster monomial, of
\(\mathbf s(v,\beta)\) therefore lifts to an ambient seed and
ambient cluster monomial. In particular,
\[
\cA_0(\mathbf s(v,\beta))\hookrightarrow
\cA_0(\mathbf s(\beta))
\xrightarrow{\sim}K_0(\mathscr C_{\DD}(b)).
\]

We next propagate the parameter vanishing of
Theorem~\ref{thm:induction}. Work inside
\(\widehat{\cA}(\Delta^{M+1})\) with its \(\dot\beta^M\)-PBW
parameters. Let \(\mathcal V_m\) be the span of the normalized global
basis elements whose first \(m\) coordinates vanish. This is a
subalgebra: by Lemma~\ref{lem:twoprod}, every index in a product with
parameters \(\mathbf a,\mathbf b\) is at most
\(\mathbf a+\mathbf b\) in the triangular order. If both parameters
have zero initial \(m\) coordinates, the same first-difference
argument as in Lemma~\ref{lem:categorical-prefix-vanishing} gives this
vanishing for every index of the product.

The normalized quantum variables of the initial seed
\(\mathbf s(v,\beta)\) belong to \(\mathcal V_m\), by
Theorem~\ref{thm:induction}. Suppose this holds for all variables in
a reached seed, and mutate a variable \(X\) to \(X'\). In the
bosonic realization the exchange relation has the form
\[
XX'=q^a\mathcal M_+ +q^b\mathcal M_-,
\qquad a,b\in\tfrac12\ZZ,
\]
where \(\mathcal M_\pm\) are products of variables of that seed.
The right-hand side lies in \(\mathcal V_m\). Both \(X\) and \(X'\)
are ambient normalized global basis elements by
Theorem~\ref{theo:isoAb}. Lemma~\ref{lem:twoprod} shows that the
maximal index occurring on the left is
\[
\bfa^{\dot\beta^M}(X)+\bfa^{\dot\beta^M}(X'),
\]
with nonzero coefficient. It must have zero first \(m\) coordinates,
because the right-hand side is supported on such indices. Since
\(X\) has this vanishing, so does \(X'\). Induction proves the
vanishing for all quantum cluster variables of the restricted seed,
and for their cluster monomials by the subalgebra property.

By Lemma~\ref{lem:categorical-prefix-vanishing}, every corresponding
real simple module belongs to \(\mathscr C_{v,\beta}\). The image of
the preceding injective homomorphism consequently lies in
\(K_0(\mathscr C_{v,\beta})\), since \(\cA_0\) is generated over
\(\ZZ\) by its cluster and frozen variables. This proves both
assertions.
\end{proof}

\begin{remark}
The preceding proof is independent of a geometric identification of
\(\mathbf s(v,\beta)\) and of an equality between its cluster and
upper cluster algebras. It establishes an embedding and the
realization of cluster monomials; it does not assert surjectivity
onto \(K_0(\mathscr C_{v,\beta})\).
\end{remark}

\noindent
\textbf{Use of ChatGPT.}
The author used OpenAI's ChatGPT to improve the language and
presentation of this paper, including the exposition of some of its
main results.


\begin{thebibliography}{CGGLSS25}

\bibitem[BH22]{bao2022total}
H.~Bao and X.~He,
\newblock Total positivity in twisted product of flag varieties,
\newblock Preprint, arXiv:2211.11168, 2022.

\bibitem[BY25]{bao2025upper}
H.~Bao and J.~Y.~Ye,
\newblock Upper cluster structure on Kac--Moody Richardson varieties,
\newblock Preprint, arXiv:2506.10382, 2025.

\bibitem[Bi24]{bi2024monoidal}
Y.~Bi,
\newblock Monoidal categorification on open Richardson varieties,
\newblock Preprint, arXiv:2409.04715, 2024.

\bibitem[Bi25]{bi2025cluster}
Y.~Bi,
\newblock On cluster structures of bosonic extensions,
\newblock Preprint, arXiv:2506.00882, 2025.

\bibitem[BZ05]{berenstein2005quantum}
A.~Berenstein and A.~Zelevinsky,
\newblock Quantum cluster algebras,
\newblock {\em Adv. Math.} \textbf{195} (2005), no.~2, 405--455.

\bibitem[CGGLSS25]{casals2025cluster}
R.~Casals, E.~Gorsky, M.~Gorsky, I.~Le, L.~Shen, and J.~Simental,
\newblock Cluster structures on braid varieties,
\newblock {\em J. Amer. Math. Soc.} \textbf{38} (2025), no.~2, 369--479.

\bibitem[FHOO23]{fujita2023isomorphisms}
R.~Fujita, D.~Hernandez, S.-j.~Oh, and H.~Oya,
\newblock Isomorphisms among quantum Grothendieck rings and cluster algebras,
\newblock Preprint, arXiv:2304.02562, 2023.

\bibitem[FZ02]{fomin2002cluster}
S.~Fomin and A.~Zelevinsky,
\newblock Cluster algebras I: foundations,
\newblock {\em J. Amer. Math. Soc.} \textbf{15} (2002), no.~2, 497--529.

\bibitem[GLSB26]{galashin2025braid}
P.~Galashin, T.~Lam, and M.~Sherman-Bennett,
\newblock Braid variety cluster structures, II: general type,
\newblock {\em Invent. Math.} \textbf{243} (2026), no.~3, 1079--1127.

\bibitem[GLS13]{geiss2013cluster}
C.~Gei{\ss}, B.~Leclerc, and J.~Schr{\"o}er,
\newblock Cluster structures on quantum coordinate rings,
\newblock {\em Selecta Math. (N.S.)} \textbf{19} (2013), no.~2, 337--397.

\bibitem[HL10]{hernandez2010cluster}
D.~Hernandez and B.~Leclerc,
\newblock Cluster algebras and quantum affine algebras,
\newblock {\em Duke Math. J.} \textbf{154} (2010), no.~2, 265--341.

\bibitem[HL16]{hernandez2016cluster}
D.~Hernandez and B.~Leclerc,
\newblock A cluster algebra approach to $q$-characters of Kirillov--Reshetikhin modules,
\newblock {\em J. Eur. Math. Soc. (JEMS)} \textbf{18} (2016), no.~5, 1113--1159.

\bibitem[Kam10]{kamnitzer2010mirkovic}
J.~Kamnitzer,
\newblock Mirkovi\'c--Vilonen cycles and polytopes,
\newblock {\em Ann. of Math. (2)} \textbf{171} (2010), no.~1, 245--294.

\bibitem[KKKO18]{kang2018monoidal}
S.-J.~Kang, M.~Kashiwara, M.~Kim, and S.-j.~Oh,
\newblock Monoidal categorification of cluster algebras,
\newblock {\em J. Amer. Math. Soc.} \textbf{31} (2018), no.~2, 349--426.

\bibitem[KKOP21]{kashiwara2021braid}
M.~Kashiwara, M.~Kim, S.-j.~Oh, and E.~Park,
\newblock Braid group action on the module category of quantum affine algebras,
\newblock {\em Proc. Japan Acad. Ser. A Math. Sci.} \textbf{97} (2021), no.~3, 13--18.

\bibitem[KKOP24a]{kashiwara2024monoidal}
M.~Kashiwara, M.~Kim, S.-j.~Oh, and E.~Park,
\newblock Monoidal categorification and quantum affine algebras II,
\newblock {\em Invent. Math.} \textbf{236} (2024), no.~2, 837--924.

\bibitem[KKOP24b]{kashiwara2024braid}
M.~Kashiwara, M.~Kim, S.-j.~Oh, and E.~Park,
\newblock Braid symmetries on bosonic extensions,
\newblock Preprint, arXiv:2408.07312, 2024.

\bibitem[KKOP25a]{kashiwara2025global}
M.~Kashiwara, M.~Kim, S.-j.~Oh, and E.~Park,
\newblock Global bases for bosonic extensions of quantum unipotent coordinate rings,
\newblock {\em Proc. Lond. Math. Soc. (3)} \textbf{131} (2025), no.~2, e70076.

\bibitem[KKOP25b]{kashiwara2025monoidal}
M.~Kashiwara, M.~Kim, S.-j.~Oh, and E.~Park,
\newblock Monoidal categorification and quantum affine algebras III,
\newblock Preprint, arXiv:2509.14552, 2025.

\bibitem[Men22]{menard2022cluster}
\'E.~M\'enard,
\newblock Cluster algebras associated with open Richardson varieties: an algorithm to compute initial seeds,
\newblock Preprint, arXiv:2201.10292, 2022.

\bibitem[OP25]{oh2025pbw}
S.-j.~Oh and E.~Park,
\newblock PBW theory for bosonic extensions of quantum groups,
\newblock {\em Int. Math. Res. Not. IMRN} \textbf{2025} (2025), no.~6, rnaf049.

\bibitem[Qin24]{qin2024analogs}
F.~Qin,
\newblock Analogs of the dual canonical bases for cluster algebras from Lie theory,
\newblock Preprint, arXiv:2407.02480, 2024.

\bibitem[SW21]{shen2021cluster}
L.~Shen and D.~Weng,
\newblock Cluster structures on double Bott--Samelson cells,
\newblock {\em Forum Math. Sigma} \textbf{9} (2021), e66.

\bibitem[WY07]{webster2007deodhar}
B.~Webster and M.~Yakimov,
\newblock A Deodhar-type stratification on the double flag variety,
\newblock {\em Transform. Groups} \textbf{12} (2007), no.~4, 769--785.

\end{thebibliography}
\end{document}